\RequirePackage{latexsym, amsmath, amstext, amssymb, amsfonts, amscd, bm, array, amsbsy, mathrsfs, amscd}
\documentclass[11pt, cmp, numbook, final, a4paper]{svjour}
\usepackage[11pt]{extsizes}
\pdfoutput=1
\usepackage{indentfirst}
\usepackage{graphics} 
\usepackage{epsfig}
\usepackage{transparent, colortbl}
\usepackage{booktabs}
\usepackage{multirow}
\usepackage{anysize}
\usepackage{latexsym}
\usepackage{pdfsync}
\usepackage[all]{xy}
\usepackage[usenames,dvipsnames]{xcolor}
\usepackage{enumerate}
\usepackage{float}
\restylefloat{table}
\usepackage{upgreek}
\usepackage{array}
\usepackage[colorlinks=true,linkcolor=black]{hyperref}
\colorlet{linkequation}{blue}
\usepackage{multicol}

\newcommand*{\SavedEqref}{}
\let\SavedEqref\eqref
\renewcommand*{\eqref}[1]{%
  \begingroup
    \hypersetup{
      linkcolor=blue,
      linkbordercolor=blue,
    }%
    \SavedEqref{#1}%
  \endgroup
}

\def\bk{\mathbf{k}}

\def\Div{\mathrm{Div}}
\def\fDiv{\mathfrak{Div}}

\DeclareSymbolFont{extraup}{U}{zavm}{m}{n}
\DeclareMathSymbol{\varheart}{\mathalpha}{extraup}{86}
\DeclareMathSymbol{\vardiamond}{\mathalpha}{extraup}{87}

\makeatletter
\def\CT@@do@color{%
  \global\let\CT@do@color\relax
        \@tempdima\wd\z@
        \advance\@tempdima\@tempdimb
        \advance\@tempdima\@tempdimc
        \kern-\@tempdimb
\transparent{0.6}%
        \leaders\vrule
                \hskip\@tempdima\@plus  1fill
        \kern-\@tempdimc
        \hskip-\wd\z@ \@plus -1fill }
\makeatother

\makeatletter
\newcommand{\thickhline}{%
    \noalign {\ifnum 0=`}\fi \hrule height 1pt
    \futurelet \reserved@a \@xhline
}
\newcolumntype{"}{@{\hskip\tabcolsep\vrule width 1pt\hskip\tabcolsep}}
\makeatother

\newtheorem{Theorem}{Theorem}[section]
\newtheorem{Lemma}[Theorem]{Lemma}
\newtheorem{Proposition}[Theorem]{Proposition}
\newtheorem{Corollary}[Theorem]{Corollary}
\newtheorem{Conjecture}[Theorem]{Conjecture}
\newtheorem{Definition}[Theorem]{Definition}
\newtheorem{Example}[Theorem]{Example}

\newcommand{\Hom}{{\rm Hom}}
\newcommand{\A}{{\Bbb{A}}}

\newcommand{\Spec}{\mathrm{Spec}}

\newcommand{\Sym}{\mathrm{Sym}}

\newcommand{\im}{\mathrm{im}}

\newcommand{\bp}{\begin{Proposition}}
\newcommand{\ep}{\end{Proposition}}
\newcommand{\bl}{\begin{Lemma}}
\newcommand{\el}{\end{Lemma}}
\newcommand{\bt}{\begin{Theorem}}
\newcommand{\et}{\end{Theorem}}
\newcommand{\bd}{\begin{Definition}}
\newcommand{\ed}{\end{Definition}}
\newcommand{\End}{\mathrm{End}}

\newcommand{\Mod}{\mathrm{Mod}}
\newcommand{\uMod}{\underline{\Mod}}

\newcommand{\Mat}{\mathrm{Mat}}

\newcommand{\eqdef}{\stackrel{{\rm def.}}{=}}

\def\rsim{\mathrm{sim}}

\DeclareFontFamily{U}{rsf}{}
\DeclareFontShape{U}{rsf}{m}{n}{<5> <6> rsfs5 <7> <8> <9> rsfs7 <10-> rsfs10}{}
\DeclareMathAlphabet\Scr{U}{rsf}{m}{n}

\def\N{\mathbb{N}}
\def\Z{\mathbb{Z}}
\def\C{\mathbb{C}}
\def\R{\mathbb{R}}

\def\rk{{\rm rk}}

\def\GL{\mathrm{GL}}

\def\supp{\mathrm{supp}\,}
\def\minsupp{\mathrm{minsupp}\,}

\def\red{\mathrm{red}}

\def\fd{\mathfrak{d}}

\def\cC{\mathcal{C}}

\def\MF{\mathrm{MF}}
\def\HMF{\mathrm{HMF}}

\newcommand{\ord}{\mathrm{ord}}
\def\essred{\mathrm{essred}}

\def\MF{\mathrm{MF}}
\def\HMF{\mathrm{HMF}}
\def\ZMF{\mathrm{ZMF}}
\def\BMF{\mathrm{BMF}}
\def\EF{\mathrm{EF}}
\def\HEF{\mathrm{HEF}}
\def\ZEF{\mathrm{ZEF}}

\def\Z{\mathbb{Z}}

\def\bm{\mathbf{m}}
\def\bmu{\boldsymbol{\mu}}

\def\loc{\mathrm{loc}}

\def\gr{\mathrm{gr}}

\def\bepsilon{\boldsymbol{\epsilon}}

\newcommand{\be}{\begin{equation*}}
\newcommand{\ee}{\end{equation*}}
\newcommand{\ben}{\begin{equation}}
\newcommand{\een}{\end{equation}}
\newcommand{\beqa}{\begin{eqnarray*}}
\newcommand{\eeqa}{\end{eqnarray*}}
\newcommand{\beqan}{\begin{eqnarray}}
\newcommand{\eeqan}{\end{eqnarray}}

\newcommand \QQ {{\mathbb Q}}

\def\2{\bar{2}}
\def\3{\bar{3}}
\def\4{\bar{4}}

\def\chS{\check{S}}

\def\Wv{\widetilde{v}}
\def\cHef{{\mathcal Hef}}
\def\ccHef{{\check{\mathcal H}ef}}
\def\mEF{\mathbf{EF}}
\def\mHEF{\mathbf{HEF}}
\def\mhef{\mathbf{hef}}

\def\chN{\check{N}}
\def\cch{\check{h}}

\newcommand{\nn}{\nonumber}

\newcommand{\id}{\mathrm{id}}

\def\O{\mathrm{O}}

\def\cP{\mathcal{P}}

\def\cT{\mathcal{T}}

\def\G_2{\mathrm{G_2}}

\def\fC{\mathfrak{C}}

\def\fI{\mathfrak{I}}

\def\Ann{\mathrm{Ann}}

\def\mf{\mathbf{f}}

\def\Ob{\mathrm{Ob}}

\def\sc{\mathrm{sc}}

\def\0{{\hat{0}}}
\def\1{{\hat{1}}}

\def\Crit{\mathrm{Crit}}

\def\O{\mathrm{O}}

\def\mod{\mathrm{mod}}

\def\bsr{\mathrm{bsr}}
\def\uHom{\underline{\Hom}}

\def\hmf{\mathrm{hmf}}
\def\zmf{\mathrm{zmf}}
\def\bmf{\mathrm{bmf}}
\def\mf{\mathrm{mf}}
\def\zef{\mathrm{zef}}
\def\hef{\mathrm{hef}}
\def\rMP{\mathrm{MP}}

\def\red{\mathrm{red}}

\newcommand{\twopartdef}[4]
{
	\left\{
		\begin{array}{ll}
			#1 & \mbox{if } #2 \\
			#3 & \mbox{if } #4
		\end{array}
	\right.
}

\newcommand{\threepartdef}[6]
{
	\left\{
		\begin{array}{ll}
			#1 & \mbox{if } #2 \\
			#3 & \mbox{if } #4 \\
                        #5 & \mbox{if } #6
		\end{array}
	\right.
}

\begin{document}

\title{Elementary matrix factorizations over B\'ezout domains}

\author{Dmitry Doryn, Calin Iuliu Lazaroiu, Mehdi Tavakol}

\institute{Center for Geometry and Physics, Institute for Basic
  Science, Pohang, Republic of Korea 37673\\
Max-Planck Institut f\"ur Mathematik, Vivatsgasse 7, 
53111 Bonn, Germany  \\
  \email{doryn@ibs.re.kr,
    calin@ibs.re.kr, mehdi@mpim-bonn.mpg.de}}

\date{}

\maketitle

\abstract{We study the homotopy category $\hef(R,W)$ (and its
$\Z_2$-graded version $\HEF(R,W)$) of elementary factorizations, where
$R$ is a B\'ezout domain which has prime elements and $W=W_0 W_c$,
where $W_0\in R^\times$ is a square-free element of $R$ and $W_c\in
R^\times$ is a finite product of primes with order at least two. In
this situation, we give criteria for detecting isomorphisms in
$\hef(R,W)$ and $\HEF(R,W)$ and formulas for the number of isomorphism
classes of objects. We also study the full subcategory $\mhef(R,W)$ of
the homotopy category $\hmf(R,W)$ of finite rank matrix factorizations
of $W$ which is additively generated by elementary factorizations. We
show that $\mhef(R,W)$ is Krull-Schmidt and we conjecture that it
coincides with $\hmf(R,W)$. Finally, we discuss a few classes of examples.}

\tableofcontents

\section*{Introduction}

\noindent The study of topological Landau-Ginzburg models \cite{LG1,LG2,lg1,lg2}
often leads to the problem of understanding the triangulated category
$\hmf(R,W)$ of finite rank matrix factorizations of an element $W\in
R$, where $R$ is a {\em non-Noetherian} commutative ring.  For
example, the category of B-type topological D-branes associated to a
holomorphic Landau-Ginzburg pair $(\Sigma,W)$ with $\Sigma$ a
non-compact Riemann surface and $W:\Sigma\rightarrow \C$ a
non-constant holomorphic function has this form with $R=\O(\Sigma)$,
the non-Noetherian ring of holomorphic functions defined on
$\Sigma$. When $\Sigma$ is connected, the ring $\O(\Sigma)$ is a
B\'ezout domain (in fact, an elementary divisor domain). In this situation, 
this problem can be reduced \cite{edd} to the study of the full subcategory
$\hef(R,W)$ whose objects are the {\em elementary factorizations},
defined as those matrix factorizations of $W$ for which the even and
odd components of the underlying supermodule have rank one. In this paper, 
we study the category $\hef(R,W)$ and the
full category $\mhef(R,W)$ of $\hmf(R,W)$ which is additively
generated by elementary matrix factorizations, for the case when $R$
is a B\'ezout domain.  We say that $W$ is {\em critically-finite} if it
is a product of a square-free element $W_0$ of $R$ with an element
$W_c\in R$ which can be written as a finite product of primes of
multiplicities strictly greater than one. When $W$ is
critically-finite, the results of this paper provide a detailed
description of the categories $\hef(R,W)$ and $\mhef(R,W)$, reducing
questions about them to the divisibility theory of $R$.

The paper is organized as follows. In Section 1, we recall some basic
facts about finite rank matrix factorizations over unital commutative
rings and introduce notation and terminology which will be used later
on. In Section 2, we study the category $\hef(R,W)$ and its
$\Z_2$-graded completion $\HEF(R,W)$ when $W$ is any non-zero element
of $R$, describing these categories in terms of the lattice of
divisors of $W$ and giving criteria for deciding when two objects are
isomorphic. We also study the behavior of these categories under
localization at a multiplicative set as well their subcategories of
primary matrix factorizations. In Section 3, we show that the additive
category $\mhef(R,W)$ is Krull-Schmidt when $R$ is a B\'ezout domain
and $W$ is a critically-finite element of $R$ and propose a few
conjectures about $\hmf(R,W)$. In Section 4, we give a formula for the
number of isomorphism classes in the categories $\HEF(R,W)$ and
$\hef(R,W)$. Finally, Section 5 discusses a few classes of
examples. Appendices \ref{app:GCD} and \ref{app:Bezout} collect some
information on greatest common denominator (GCD) domains and B\'ezout
domains.

\paragraph{Notations and conventions}
The symbols $\0$ and $\1$ denote the two elements of the field
$\Z_2=\Z/2\Z$, where $\0$ is the zero element. Unless otherwise
specified, all rings considered are unital and commutative.
Given a cancellative Abelian monoid $(M,\cdot)$, we say that an
element $x\in M$ divides $y\in M$ if there exists $q\in M$ such that
$y=qx$. In this case, $q$ is uniquely determined by $x$ and $y$ and we
denote it by $q=x/y$ or $\frac{x}{y}$.

Let $R$ be a unital commutative ring. The set of non-zero elements of
$R$ is denoted by $R^\times\eqdef R\setminus \{0\}$, while the
multiplicative group of units of $R$ is denoted by $U(R)$. The Abelian
categories of all $R$-modules is denoted $\Mod_R$, while the Abelian
category of finitely-generated $R$-modules is denoted $\mod_R$. Let
$\Mod_R^{\Z_2}$ denote the category of $\Z_2$-graded modules and outer
(i.e. even) morphisms of such and $\uMod_R^{\Z_2}$ denote the category
of $\Z_2$-graded modules and inner morphisms of such.  By definition,
an {\em $R$-linear category} is a category enriched in the monoidal
category $\Mod_R$ while a {\em $\Z_2$-graded $R$-linear category} is a
category enriched in the monoidal category $\uMod_R^{\Z_2}$. With this
definition, a linear category is pre-additive, but it need not admit
finite bi-products (direct sums).  For any $\Z_2$-graded $R$-linear
category $\cC$, the {\em even subcategory} $\cC^\0$ is the $R$-linear
category obtained from $\cC$ by keeping only the even morphisms.

For any unital integral domain $R$, let $\sim$ denote the equivalence
relation defined on $R^\times$ by association in divisibility:
\be
x\sim y ~~\mathrm{iff}~~\exists \gamma\in U(R): y=\gamma x~~.
\ee
The set of equivalence classes of this relation coincides with the set
$R^\times/U(R)$ of orbits for the obvious multiplicative action of
$U(R)$. Since $R$ is a commutative domain, the quotient
$R^\times/U(R)$ inherits a multiplicative structure of cancellative
Abelian monoid. For any $x\in R^\times$, let $(x)\in R^\times/U(R)$
denote the equivalence class of $x$ under $\sim$. Then for any $x,y\in
R^\times$, we have $(xy)=(x)(y)$. The monoid $R^\times/U(R)$ can also
be described as follows. Let $G_+(R)$ be the set of non-zero principal
ideals of $R$. If $x,y$ are elements of $R^\times$, we have $\langle
x\rangle \langle y\rangle =\langle xy\rangle$, so the product of
principal ideals corresponds to the product of the multiplicative
group $R^\times$ and makes $G_+(R)$ into a cancellative Abelian monoid
with unit $\langle 1\rangle=R$. Notice that $G_+(R)$ coincides with
the positive cone of the group of divisibility (see Subsection
\ref{subsec:div}) $G(R)$ of $R$, when the latter is viewed as an
Abelian group ordered by reverse inclusion. The monoids
$R^\times/U(R)$ and $G_+(R)$ can be identified as follows.  For any
$x\in R^\times$, let $\langle x \rangle\in G_+(R)$ denote the
principal ideal generated by $x$. Then $\langle x \rangle$ depends
only on $(x)$ and will also be denoted by $\langle (x)\rangle$. This
gives a group morphism $\langle~\rangle: R^\times/U(R)\rightarrow
G_+(R)$.  For any non-zero principal ideal $I\in G_+(R)$, the set of
all generators $x$ of $I$ is a class in $R^\times/U(R)$ which we
denote by $(I)$; this gives a group morphism $(~):G_+(R)\rightarrow
R^\times/U(R)$.  For all $x\in R^\times$, we have $(\langle
x\rangle)=(x)$ and $\langle (x)\rangle =\langle x \rangle$, which
implies that $\langle~\rangle$ and $(~)$ are mutually inverse group
isomorphisms.

If $R$ is a GCD domain (see Appendix \ref{app:GCD}) and $x_1,\ldots,
x_n$ are elements of $R$ such that $x_1\ldots x_n\neq 0$, let $d$ be
any greatest common divisor (gcd) of $x_1,\ldots, x_n$. Then $d$ is
determined by $x_1,\ldots, x_n$ up to association in divisibility and
we denote its equivalence class by $(x_1,\ldots, x_n)\in
R^\times/U(R)$. The principal ideal $\langle d\rangle =\langle
(x_1,\ldots, x_n)\rangle \in G_+(R)$ does not depend on the choice of
$d$. The elements $x_1,\ldots, x_n$ also have a least common multiple
(lcm) $m$, which is determined up to association in divisibility and
whose equivalence class we denote by $[x_1,\ldots, x_n]\in
R^\times/U(R)$. For $n=2$, we have:
\be
[x_1,x_2]=\frac{(x_1)(x_2)}{(x_1,x_2)}~~.
\ee
If $R$ is a B\'ezout domain (see Appendix \ref{app:Bezout}), then we
have $\langle x_1,\ldots, x_n\rangle\eqdef \langle (x_1,\ldots, x_n)\rangle=\langle x_1\rangle +\ldots
+\langle x_n\rangle$, 
so the gcd operation transfers the operation given by taking the finite sum of
principal ideals from $G_+(R)$ to $R^\times/U(R)$ through the
isomorphism of groups described above. In this case, we have
$(x_1,\ldots,x_n)=(\langle x_1,\ldots, x_n\rangle)$. We also have
$\langle [x_1,\ldots, x_n]\rangle=\cap_{i=1}^n\langle x_i\rangle$ and
hence $[x_1,\ldots, x_n]=(\cap_{i=1}^n\langle x_i\rangle)$.  Thus the
lcm corresponds to the finite intersection of principal ideals.

\section{Matrix factorizations over an integral domain} 

\

\noindent Let $R$ be an integral domain and $W\in R^\times$ be
a non-zero element of $R$.

\subsection{Categories of matrix factorizations}

We shall use the following notations: 
\begin{enumerate}
\itemsep 0.0em
\item $\MF(R,W)$ denotes the $R$-linear and $\Z_2$-graded differential
  category of $R$-valued matrix factorizations of $W$ of finite
  rank. The objects of this category are pairs $a=(M,D)$, where $M$ is
  a free $\Z_2$-graded $R$-module of finite rank and $D$ is an odd
  endomorphism of $M$ such that $D^2=W\id_M$. For any objects
  $a_1=(M_1,D_1)$ and $a_2=(M_2,D_2)$ of $\MF(R,W)$, the $\Z_2$-graded
  $R$-module of morphisms from $a_1$ to $a_2$ is given by the inner
  $\Hom$:
\be
\Hom_{\MF(R,W)}(a_1,a_2)=\uHom_R(M_1,M_2)=\Hom_R^\0(M_1,M_2)\oplus \Hom_R^\1(M_1,M_2)~~,  
\ee
endowed with the differential $\fd_{a_1,a_2}$ determined uniquely by the condition:
\be
\fd_{a_1,a_2}(f)=D_2\circ f-(-1)^\kappa f\circ D_1~~,~~\forall f\in \Hom_R^\kappa(M_1,M_2)~~,
\ee
where $\kappa\in \Z_2$.
\item $\ZMF(R,W)$ denotes the $R$-linear and $\Z_2$-graded cocycle category of $\MF(R,W)$. 
This has the same objects as $\MF(R,W)$ but morphism spaces given by: 
\be
\Hom_{\ZMF(R,W)}(a_1,a_2)\eqdef \{f\in \Hom_{\MF(R,W)}(a_1,a_2)|\fd_{a_1,a_2}(f)=0\}~~.
\ee
\item $\BMF(R,W)$ denotes the $R$-linear and $\Z_2$-graded coboundary
  category of $\MF(R,W)$, which is an ideal in $\ZMF(R,W)$.  This has
  the same objects as $\MF(R,W)$ but morphism spaces given by:
\be
\Hom_{\BMF(R,W)}(a_1,a_2)\eqdef \{\fd_{a_1,a_2}(f)|f\in \Hom_{\MF(R,W)}(a_1,a_2)\}~~.
\ee
\item $\HMF(R,W)$ denotes the $R$-linear and $\Z_2$-graded total cohomology category of
  $\MF(R,W)$. This has the same objects as $\MF(R,W)$ but morphism
  spaces given by:
\be
\Hom_{\HMF(R,W)}(a_1,a_2)\eqdef \Hom_{\ZMF(R,W)}(a_1,a_2)/\Hom_{\BMF(R,W)}(a_1,a_2)~~.
\ee
\item The subcategories of $\MF(R,W)$, $\ZMF(R,W)$, $\BMF(R,W)$ and
  $\HMF(R,W)$ obtained by restricting to morphisms of even degree are
  denoted respectively by $\mf(R,W)\eqdef \MF^\0(R,W)$,
  $\zmf(R,W)\eqdef \ZMF^\0(R,W)$, $\bmf(R,W)\eqdef \BMF^\0(R,W)$ and
  $\hmf(R,W)\eqdef \HMF^\0(R,W)$.
\end{enumerate}
The categories $\MF(R,W)$, $\BMF(R,W)$ and $\ZMF(R,W)$ admit double
direct sums (and hence all finite direct sums of at least two
elements) but do not have zero objects. On the other hand, the
category $\HMF(R,W)$ is additive, the matrix factorization
$\left[\begin{array}{cc} 0 &1 \\ W & 0\end{array}\right]$ being a zero
object. Finally, it is well-known that the category $\hmf(R,W)$ is
triangulated (see \cite{Langfeldt} for a detailed treatment).

\

\noindent For later reference, recall that the biproduct (direct sum)
of $\MF(R,W)$ is defined as follows:

\

\begin{Definition}
Given two matrix factorizations $a_i=(M_i,D_i)$, $(i=1,2)$ of $W\in
R$, their direct sum $a_1 \oplus a_2$ is the matrix factorization
$a=(M,D)$ of $W$, where $M\eqdef M^\0 \oplus M^\1$ and $D\eqdef
\left[\begin{array}{cc} 0 & v\\ u & 0\end{array}\right]$, with:
\be
M^\kappa=M_1^\kappa \oplus M_2^\kappa ~\forall \kappa\in \Z_2 ~~\mathrm{and}~~ u=\left[\begin{array}{cc} u_1 & 0 \\ 0 & u_2\end{array}\right]~,~ v=\left[\begin{array}{cc} v_1 & 0 \\ 0 & v_2\end{array}\right]~~.
\ee
Given a third matrix factorization $a_3=(M_3,D_3)$ of $W$ and two
morphisms $f_i\in \Hom_{MF(R,W)}(a_i,a_3)=\uHom_R(a_i,a_3)$ $(i=1,2)$
in $\MF(R,W)$, their direct sum of $f_1\oplus f_2\in
\Hom_{\MF(R,W)}(a_1\oplus a_2,a_3)=\uHom_R(a_1\oplus a_2,a_3)$ is the
ordinary direct sum of the $R$-module morphisms $f_1$ and $f_2$.
\end{Definition}

\

\noindent As a consequence, $\MF(R,W)$ admits all finite but non-empty
direct sums. The following result is elementary:

\

\begin{Lemma}
The following statements hold:
\begin{enumerate}
\item The subcategories $\ZMF(R,W)$ and $\BMF(R,W)$ of $\MF(R,W)$ are
  closed under finite direct sums (but need not have zero objects). 
\item The direct sum induces a well-defined biproduct (which is
  again denoted by $\oplus$) on the $R$-linear categories $\HMF(R,W)$
  and $\hmf(R,W)$.
 \item $(\HMF(R,W),\oplus)$ and $(\hmf(R,W),\oplus)$ are additive
  categories, a zero object in each being given by any of the
  elementary factorizations $e_1$ and $e_W$, which are isomorphic to
  each other in $\hmf(R,W)$. In particular, any finite direct sum of
  the elementary factorizations $e_1$ and $e_W$ is a zero object in
  $\HMF(R,W)$ and in $\hmf(R,W)$.
\end{enumerate}
\end{Lemma}

\subsection{Reduced rank and matrix description}

Let $a=(M,D)$ be an object of $\MF(R,W)$, where $M=M^\0\oplus
M^\1$. Taking the supertrace in the equation $D^2=W\id_M$ and using
the fact that $W\neq 0$ shows that $\rk M^\0=\rk M^\1$. We call this
natural number the {\em reduced rank} of $a$ and denote it by
$\rho(a)$; we have $\rk M=2\rho(a)$. Choosing a homogeneous basis of
$M$ (i.e. a basis of $M^\0$ and a basis of $M^\1$) gives an
isomorphism of $R$-supermodules $M\simeq R^{\rho|\rho}$, where
$\rho=\rho(a)$ and $R^{\rho|\rho}$ denotes the $R$-supermodule with
$\Z_2$-homogeneous components
$(R^{\rho|\rho})^\0=(R^{\rho|\rho})^\1=R^{\oplus \rho}$. This
isomorphism allows us to identify $D$ with a square matrix of size
$2\rho(a)$ which has block off-diagonal form:
\be
D=\left[\begin{array}{cc} 0 & v\\ u & 0\end{array}\right]~~,
\ee
where $u$ and $v$ are square matrices of size $\rho(a)$ with entries in
$R$. The condition $D^2=W\id_M$ amounts to the relations: 
\ben
\label{uvrels}
uv=vu=W I_\rho~~,
\een
where $I_\rho$ denotes the identity matrix of size $\rho$.  Since
$W\neq 0$, these conditions imply that the matrices $u$ and $v$ have
maximal rank\footnote{To see this, it suffices to consider
  equations \eqref{uvrels} in the field of fractions of $R$.}:
\be
\rk u=\rk v=\rho~~.
\ee
Matrix factorizations for which $M=R^{\rho|\rho}$ form a dg subcategory
of $\MF(R,W)$ which is {\em essential} in the sense that it is
dg-equivalent with $\MF(R,W)$. Below, we often tacitly identify
$\MF(R,W)$ with this essential subcategory and use similar
identifications for $\ZMF(R,W)$, $\BMF(R,W)$ and $\HMF(R,W)$.

Given two matrix factorizations $a_1=(R^{\rho_1|\rho_1}, D_1)$ and
$a_2=(R^{\rho_2|\rho_2}, D_2)$ of $W$, write
$D_i=\left[\begin{array}{cc} 0 & v_i\\ u_i & 0\end{array}\right]$,
with $u_i,v_i\in \Mat(\rho_i,\rho_i,R)$. Then:
\begin{itemize}
\itemsep 0.0em
\item An even morphism $f\in
\Hom_{\MF(R,W)}^\0(a_1,a_2)$ has the matrix form:
\be
f=\left[\begin{array}{cc}
    f_{\0\0} & 0\\ 0 & f_{\1\1}\end{array}\right]
\ee
with $f_{\0\0}, f_{\1\1}\in \Mat(\rho_1,\rho_2,R)$ and we have:
\be
\fd_{a_1,a_2}(f)=D_2\circ f-f\circ D_1=\left[\begin{array}{cc} 0 & v_2\circ f_{\1\1}-f_{\0\0}\circ v_1\\ u_2\circ f_{\0\0}-f_{\1\1}\circ u_1 &
    0\end{array}\right]~~;
\ee
\item An odd morphism $g\in \Hom^\1_{\MF(R,W)}(a_1,a_2)$
has the matrix form: 
\be
g=\left[\begin{array}{cc} 0 & g_{\1\0}\\ g_{\0\1} &
    0\end{array}\right]
\ee 
with $g_{\1\0}, g_{\0\1}\in \Mat(\rho_1,\rho_2,R)$ and we have:
\be
\fd_{a_1,a_2}(g)=D_2\circ g+g\circ D_1=\left[\begin{array}{cc} v_2\circ g_{\0\1}+g_{\1\0}\circ u_1 & 0 \\  0 &
  u_2\circ g_{\1\0}+g_{\0\1}\circ v_1  \end{array}\right]~~.
\ee
\end{itemize}

\begin{remark}
The cocycle condition $\fd_{a_1,a_2}(f)=0$ satisfied by an even morphism $f\in \Hom_{\ZMF(R,W)}^\0(a_1,a_2)$ 
amounts to the system: 
\beqa
\begin{cases}
 v_2\circ f_{\1\1}=f_{\0\0}\circ v_1\\
 u_2\circ f_{\0\0}=f_{\1\1}\circ u_1~~,
\end{cases} 
\eeqa
which in turn amounts to any of the following equivalent conditions:
\be
f_{\1\1}=\frac{u_2\circ f_{\0\0}\circ v_1}{W}\Longleftrightarrow f_{\0\0}=\frac{v_2\circ f_{\1\1}\circ u_1}{W}
~~.
\ee
Similarly, the cocycle condition $\fd_{a_1,a_2}(g)=0$ defining an
odd morphism $g\in \Hom^\1_{\ZMF(R,W)}(a_1,a_2)$ amounts to the system:
\beqa
\begin{cases}
 v_2\circ g_{\0\1}+g_{\1\0}\circ u_1=0\\
 u_2\circ g_{\1\0}+g_{\0\1}\circ v_1=0~~,
\end{cases}
\eeqa
which in turn amounts to any of the following equivalent conditions: 
\be
g_{\1\0}=-\frac{v_2\circ g_{\0\1}\circ v_1}{W}\Longleftrightarrow g_{\0\1}=-\frac{u_2\circ g_{\1\0}\circ u_1}{W}~~.
\ee
\end{remark}

\subsection{Strong isomorphism}

Recall that $\zmf(R,W)$ denotes the even subcategory of $\ZMF(R,W)$. This category 
admits non-empty finite direct sums but does not have a zero object. 

\

\begin{Definition}
Two matrix factorizations $a_1$ and $a_2$ of $W$ over $R$ are called {\em
  strongly isomorphic} if they are isomorphic in the category
$\zmf(R,W)$.
\end{Definition}

\

\noindent It is clear that two strongly isomorphic factorizations are also isomorphic 
in $\hmf(R,W)$, but the converse need not hold. 

\

\begin{Proposition}
\label{prop:strong_iso1}
Let $a_1=(R^{\rho_1|\rho_1},D_1)$ and $a_2=(R^{\rho_2|\rho_2},D_2)$ be
two matrix factorizations of $W$ over $R$, where
$D_i=\left[\begin{array}{cc} 0 & v_i\\ u_i & 0\end{array}\right]$.
Then the following statements are equivalent:
\begin{enumerate}[(a)]
\itemsep 0.0em
\item $a_1$ and $a_2$ are strongly isomorphic.
\item $\rho_1=\rho_2$ (a quantity which we denote by $\rho$) and there
  exist invertible matrices $A,B\in \GL(\rho,R)$ such that one (and
  hence both) of the following equivalent conditions is satisfied:
\begin{enumerate}[1.]
\item $v_2=Av_1B^{-1}$,
\item $u_2=Bu_1A^{-1}$.                                             
\end{enumerate}
\end{enumerate}
\end{Proposition}

\begin{proof}
$a_1$ and $a_2$ are strongly isomorphic iff there exists $U\in
  \Hom_{\zmf(R,W)}(a_1,a_2)$ which is an isomorphism in
  $\zmf(R,W)$. Since $U$ is an even morphism in the cocycle
  category, we have: 
\ben
\label{DS}
UD_1=D_2U~~.  
\een 
The condition that $U$ be even allows us to identify it with a matrix
of the form $U=\left[\begin{array}{cc} A & 0\\ 0 &
    B\end{array}\right]$, while invertibility of $U$ in $\zmf(R,W)$
amounts to invertibility of the matrix $U$, which in turn means that
$A$ and $B$ are square matrices (thus $\rho_1=\rho_2=\rho$) 
belonging to $\GL(\rho,R)$. Thus relation \eqref{DS} reduces to either of
conditions 1. or 2., which are equivalent since $v_1u_1=u_1v_1=W I_\rho$
and $u_2v_2=v_2u_2=W I_\rho$. \qed
\end{proof}

\subsection{Critical divisors and the critical locus of $W$}

\

\

\begin{Definition}
A divisor $d$ of $W$ which is not a unit is called {\em critical} if
$d^2|W$. 
\end{Definition}

\

\noindent Let:
\be
\fC(W)\eqdef \big\{d\in R~\big|~ d^2|W \big\}
\ee
be the set of all critical divisors of $W$. The ideal: 
\ben
\label{IW}
\fI_W\eqdef\cap_{d\in \fC(W)} \langle d \rangle
\een
is called the {\em critical ideal} of $W$. Notice that $\fI_W$ consists
of those elements of $R$ which are divisible by all critical divisors
of $W$. In particular, we have $(W)\subset \fI_W$ and hence there exists a
unital ring epimorphism $R/(W)\rightarrow R/\fI_W$.

\

\begin{Definition}
A {\em critical prime divisor} of $W$ is a prime element $p\in R$ such
that $p^2|W$. The {\em critical locus} of $W$ is the subset of
$\Spec(R)$ consisting of the principal prime ideals of $R$ generated
by the critical prime divisors of $W$:
\be
\Crit(W)\eqdef \big\{\langle p \rangle\in  \Spec(R) ~\big| ~ p^2|W \big\}~~.
\ee
\end{Definition}

\subsection{Critically-finite elements}

Let $R$ be a B\'ezout domain. Then $R$ is a GCD domain, hence
irreducible elements of $R$ are prime.  This implies that any
factorizable element\footnote{I.e. an element of $R$ which has a
  finite factorization into irreducibles.} of $R$ has a unique prime
factorization up to association in divisibility.

\

\begin{Definition}\label{def:crit-fin}
A non-zero non-unit $W$ of $R$ is called: 
\begin{itemize}
\itemsep 0.0em
\item {\em non-critical}, if $W$ has no critical divisors;
\item {\em critically-finite} if it has a factorization of the form: 
\ben
\label{Wcritform}
W=W_0 W_c~~\mathrm{with}~~W_c=p_1^{n_1}\ldots p_N^{n_N}~~,
\een
where $n_j\geq 2$, $p_1,\ldots, p_N$ are critical prime divisors of
$W$ (with $p_i\not \sim p_j$ for $i\neq j$) and $W_0$ is non-critical
and coprime with $W_c$. 
\end{itemize}
\end{Definition}
Notice that the elements $W_0$, $W_c$ and $p_i$ in the factorization
\eqref{Wcritform} are determined by $W$ up to association, while the
integers $n_i$ are uniquely determined by $W$. The factors $W_0$ and
$W_c$ are called respectively the {\em non-critical} and {\em
  critical} parts of $W$. The integers $n_i\geq 2$ are called the {\em
  orders} of the critical prime divisors $p_i$.

\

\noindent For a critically-finite element $W$ with decomposition \eqref{Wcritform}, we have: 
\be
\Crit(W)=\{\langle p_1\rangle,\ldots, \langle p_N\rangle\}~~\mathrm{and}~~\fI_W=\langle W_\red\rangle~~,
\ee
where\footnote{The notation $\lfloor x\rfloor\in \Z$ indicates the integral
part of a real number $x\in\R$.}: 
\be
W_\red\eqdef p_1^{\lfloor\frac{n_1}{2}\rfloor}\ldots p_N^{\lfloor\frac{n_N}{2}\rfloor}
\ee
is called the {\em reduction} of $W$. Notice that $W_\red$ is determined up to 
association in divisibility.

\subsection{Two-step factorizations of $W$}

Recall that a {\em two-step factorization} (or two-step multiplicative
partition) of $W$ is an ordered pair $(u,v)\in R\times R$ such that
$W=uv$. In this case, the divisors $u$ and $v$ are called {\em
  $W$-conjugate}.  The {\em transpose} of $(u,v)$ is the ordered pair
$(v,u)$ (which is again a two-step factorization of $W$), while the {\em
  opposite transpose} is the ordered pair $\sigma(u,v)=(-v,-u)$. This
defines an involution $\sigma$ of the set $\rMP_2(W)$ of two-step
factorizations of $W$.  The two-step factorizations $(u,v)$ and
$(u',v')$ are called {\em similar} (and we write $(u,v)\sim (u',v')$)
if there exists $\gamma\in U(R)$ such that $u'=\gamma u$ and
$v'=\gamma^{-1} v$. We have $\sigma(u,v)\sim (v,u)$.

\

\begin{Definition}
The {\em support} of a two-step factorization $(u,v)$ of $W$ is the principal ideal $\langle u,v\rangle\in G_+(R)$.
\end{Definition}

\

\noindent Let $d$ be a gcd of $u$ and $v$. Since $W=uv=d^2 u_1v_1$
(where $u_1\eqdef u/d$, $v_1\eqdef v/d$), it is clear that $d$ is a critical divisor of
$W$. Notice that the opposite transpose of the two step factorization
$(u,v)$ has the same support as $(u,v)$.

\subsection{Elementary matrix factorizations}
\label{subsec:elem}

\

\

\begin{Definition}
A matrix factorization $a=(M,D)$ of $W$ over $R$ is called {\em
  elementary} if it has unit reduced rank, i.e. if $\rho(a)=1$.
\end{Definition}

\

\noindent Any elementary factorization is strongly isomorphic to one
of the form $e_v\eqdef (R^{1|1},D_v)$, where $v$ is a divisor of $W$
and $D_v\eqdef \left[\begin{array}{cc} 0 & v\\ u &
    0 \end{array}\right]$, with $u\eqdef W/v\in R$. Let $\EF(R,W)$
denote the full subcategory of $\MF(R,W)$ whose objects are the
elementary factorizations of $W$ over $R$. Let $\ZEF(R,W)$ and
$\HEF(R,W)$ denote respectively the cocycle and total cohomology
categories of $\EF(R,W)$.  We also use the notations $\zef(R,W)\eqdef
\ZEF^\0(R,W)$ and $\hef(R,W)\eqdef \HEF^\0(R,W)$. Notice that an
elementary factorization is indecomposable in $\zmf(R,W)$, but it need
not be indecomposable in the triangulated category $\hmf(R,W)$.

The map $\Phi:\Ob\EF(M,W)\rightarrow \rMP_2(W)$ which sends $e_v$ to the
ordered pair $(u,v)$ is a bijection. The suspension of $e_v$ is given
by $\Sigma e_v=e_{-u}=(R^{1|1},D_{-u})$, since:
\be
D_{-u}=\left[\begin{array}{cc} 0 & -u\\ -v & 0 \end{array}\right]~~.
\ee
In particular, $\Sigma e_v$ corresponds to the opposite transpose
$\sigma(u,v)$ and we have:
\be
\Phi\circ \Sigma=\sigma\circ \Phi~~.
\ee
Hence $\Sigma$ preserves the subcategory $\EF(M,W)$ of $\MF(R,W)$ 
and the subcategories $\HEF(R,W)$ and $\hef(R,W)$ of $\HMF(R,W)$ 
and $\hmf(R,W)$. This implies that $\HEF(R,W)$ is equivalent with the 
graded completion $\gr_\Sigma \hef(R,W)$. We thus have natural isomorphisms:
\beqan
\label{graded}
\Hom^\1_{\HEF(R,W)}(e_{v_1},e_{v_2}) &\simeq_R&  \Hom_{\hef(R,W)}(e_{v_1},\Sigma e_{v_2})=
\Hom_{\hef(R,W)}(e_{v_1},e_{-u_2})\nn  ~~,\\
\Hom^\1_{\HEF(R,W)}(e_{v_1},e_{v_2}) &\simeq_R&  \Hom_{\hef(R,W)}(\Sigma e_{v_1},e_{v_2})=\Hom_{\hef(R,W)}(e_{-u_1},e_{v_2}) ~~,
\eeqan
for any divisors $v_1,v_2$ of $W$, where $u_1=W/v_1$ and $u_2=W/v_2$.

\

\begin{Definition}\label{def:110}
The {\em support} of an elementary matrix factorization $e_v$ is the
ideal of $R$ defined through:
\be
\supp(e_v)\eqdef \supp (\Phi(e_v))=\langle v,W/v\rangle ~~.
\ee
Notice that this ideal is generated by any gcd $d$ of $v$ and $W/v$ and that $d$ is a critical divisor of $W$. 
\end{Definition}

\

\noindent We will see later that an elementary factorization is trivial iff its support equals $R$. 

\

\begin{Definition}
Two elementary matrix factorizations $e_{v_1}$ and $e_{v_2}$ of $W$
are called \emph{similar} if $v_1\sim v_2$ or equivalently $u_1\sim
u_2$.  This amounts to existence of a unit $\gamma\in U(R)$ such
that $v_2=\gamma v_1$ and $u_2=\gamma^{-1}u_1$.
\end{Definition}

\

\begin{Proposition}
\label{prop:estrong}
Two elementary factorizations $e_{v_1}$ and $e_{v_2}$ are strongly
isomorphic iff they are similar. In particular, strong isomorphism 
classes of elementary factorization are in bijection with the 
set of those principal ideals of $R$ which contain $W$. 
\end{Proposition}

\begin{proof}
Suppose that $e_{v_1}$ and $e_{v_2}$ are strongly isomorphic. By
Proposition \ref{prop:strong_iso1}, there exist units $x,y\in U(R)$
such that $v_2=xv_1 y^{-1}$ and $u_2=y u_1 x^{-1}$, where $u_i\eqdef
W/v_i$. Setting $\gamma\eqdef xy^{-1}$ gives $v_1=\gamma v_1$ and
$u_2=\gamma^{-1} u_1$, hence $e_{v_1}$ and $e_{v_2}$ are
similar. Conversely, suppose that $e_{v_1}\sim e_{v_2}$. Then there
exists a unit $\gamma\in U(R)$ such that $v_2=\gamma v_1$ and
$u_2=\gamma^{-1}u_1$. Setting $x=\gamma$ and $y=1$ gives $v_2=xv_1
y^{-1}$ and $u_2=yu_1 x^{-1}$, which shows that $e_{v_1}$ and
$e_{v_2}$ are strongly isomorphic upon using Proposition
\ref{prop:strong_iso1}. The map which sends the strong isomorphism class 
of $e_v$ to the principal ideal $(v)$ gives the bijection stated. 
\qed
\end{proof}

\

\noindent It is clear that $e_{v_1}$ and $e_{v_2}$ are similar iff the
corresponding two-step factorizations $(v_1,u_1)$ and $(v_2,u_2)$ of
$W$ are similar. Since any strong isomorphism induces an isomorphism
in $\hef(R,W)$, it follows that similar elementary factorizations are
isomorphic in $\hef(R,W)$.

\subsection{The categories $\mHEF(R,W)$ and $\mhef(R,W)$}

Let $\mEF(R,W)$ denote the smallest full $R$-linear subcategory of
$\MF(R,W)$ which contains all objects of $\EF(R,W)$ and is closed
under finite direct sums. It is clear that $\mEF(R,W)$ is a full dg
subcategory of $\MF(R,W)$. Let $\mHEF(R,W)$ denote the total
cohomology category of $\mEF(R,W)$. Let $\mhef(R,W)\eqdef
\mHEF^\0(R,W)$ denote the subcategory obtained from $\mHEF(R,W)$ by
keeping only the even morphisms. Notice that $\mhef(R,W)$ coincides
with the smallest full subcategory of $\hmf(R,W)$ which contains all
elementary factorizations of $W$.

\section{Elementary matrix factorizations over a B\'ezout domain}

\noindent Throughout this section, let $R$ be a B\'ezout domain and $W$
be a non-zero element of $R$. 

\subsection{The subcategory of elementary factorizations}

Let $v_1,v_2$ be divisors of $W$ and $e_1:=e_{v_1}$, $e_2:=e_{v_2}$ be the
corresponding elementary matrix factorizations of $W$. Let $u_1\eqdef
W/v_1$, $u_2=W/v_2$. Let $a$ be a gcd of $v_1$ and $v_2$. Define: 
\ben
\label{param}
b\eqdef v_1/a~,~c\eqdef v_2/a~,~d\eqdef \frac{W}{abc}~,~a'\eqdef
a/s~,~d'\eqdef d/s~~,
\een 
where $s$ is a gcd of $a$ and $d$. Then $a=a's$ and $d=d's$ with
$(a',d')=(1)=(b,c)$ and $W=abcd=s^2 a'bcd'$. In particular, $s$ is a
critical divisor of $W$. Moreover:
\ben
\label{vu}
v_1=ab=sa'b~~,~~v_2=ac=sa'c~~,~~u_1=cd=scd'~~,~~u_2=bd=sbd'
\een
and we have: 
\ben
\label{ds}
(d)= (u_1,u_2)~,~(s)=(v_1,v_2,u_1,u_2)~~.
\een
Notice the following relations in the cancellative monoid $R^\times/U(R)$:
\ben
\begin{aligned}
\label{gcds}
(v_1,v_2)&=(sa')~,~(u_1,u_2)=(s)(d')~,~(u_1,v_1)=(s)(a',c)(b,d')~~,\\~(u_1,v_2)&=(s)(c)~,~(u_2,v_1)=(s)(b)~,~(u_2,v_2)=(s)(a',b)(c,d')~~.
\end{aligned}
\een
In this notation:
\be
D_{v_1}=\left[\begin{array}{cc} 0 & v_1\\ u_1 & 0 \end{array}\right] = \left[\begin{array}{cc} 0 & ab\\ cd & 0 \end{array}\right]=s\left[\begin{array}{cc} 0 & a'b\\ cd' & 0 \end{array}\right]~~\text{and}~~
D_{v_2}=\left[\begin{array}{cc} 0 & v_2\\ u_2 & 0 \end{array}\right] = \left[\begin{array}{cc} 0 & ac\\ bd & 0 \end{array}\right]=s\left[\begin{array}{cc} 0 & a'c\\ bd' & 0 \end{array}\right]~~.
\ee
For $f\in \Hom_{\MF(R,W)}^\0(e_1,e_2)=\Hom^\0_R(R^{1|1},R^{1|1})$ and $g\in \Hom_{\MF(R,W)}^\1(e_1,e_2)=\Hom^\1_R(R^{1|1},R^{1|1})$, we have: 
\ben
\label{fdelem}
\fd_{e_1,e_2}(f)=(c f_{\1\1}-f_{\0\0} b) \left[\begin{array}{cc} 0 & a\\ -d& 0 \end{array}
\right]~\mathrm{and}~\fd_{e_1,e_2}(g)=(a g_{\0\1}+g_{\1\0} d) \left[\begin{array}{cc} c & 0\\ 0& b \end{array}\right]~~.
\een

\begin{remark}
Relations \eqref{param} and \eqref{gcds} imply the following equalities in the cancellative monoid $R^\times/U(R)$:
\ben
\label{alpha_alt}
(s)= \frac{(u_1,v_2)}{\left((u_1,v_2),\frac{v_2}{(v_1,v_2)}\right)}=\frac{(u_2,v_1)}{\left((u_2,v_1),\frac{v_1}{(v_1,v_2)}\right)}~~.
\een
\end{remark}

\subsubsection{Morphisms in $\HEF(R,W)$}

Let $\Mat(n,R^\times/U(R))$ denote the set of square matrices of size
$n$ with entries from the multiplicative semigroup $R/U(R)$.  Any
matrix $S\in \Mat(n,R^\times/U(R))$ can be viewed as an equivalence
class of matrices $A\in \Mat(n,R^\times)$ under the equivalence
relation:
\ben
\label{simn}
A\sim_n B~~\mathrm{iff}~~\forall i,j\in \{1,\ldots, n\}:\exists q_{ij}\in U(R)~\mathrm{such~that}~ B_{ij}=q_{ij} A_{ij}~~.
\een

\

\begin{Proposition}
\label{prop:HomZMF}
With the notations above, we have:
\begin{enumerate}
\itemsep 0.0em
\item $\Hom_{\ZMF(R,W)}^\0(e_{1},e_{2})$ is the free $R$-module
  of rank one generated by the matrix:
\be
\epsilon_\0(v_1,v_2)\eqdef \left[\begin{array}{cc} c & 0 \\0 & b \end{array}\right]\in \left[\begin{array}{cc} \frac{(v_2)}{(v_1,v_2)} & 0 \\ 0 & \frac{(v_1)}{(v_1,v_2)}
\end{array}\right]\eqdef \bepsilon_\0(v_1,v_2)~~,
\ee 
where the matrix $\bepsilon_\0(v_1,v_2)\in \Mat(2,R/U(R))$ in the
right hand side is viewed as an equivalence class under the relation
\eqref{simn}.
\item $\Hom_{\ZMF(R,W)}^\1(e_{1},e_{2})$ is the free $R$-module of
  rank one generated by the matrix:
\be
\epsilon_\1(v_1,v_2; W)\eqdef \left[\begin{array}{cc} 0 & a' \\-d' & 0 \end{array}\right]\in \left[\begin{array}{cc} 0 & \frac{(v_2)}{(u_1,v_2)} \\-\frac{(u_1)}{(u_1,v_2)} & 0\end{array}\right]\eqdef \bepsilon_\1(v_1,v_2; W)~~
\ee
and we have $\bepsilon_\1(v_1,v_2; W)=\bepsilon_\1(v_2,v_1; W)$ in $\Mat(2,R/U(R))$. 
\end{enumerate}
\end{Proposition}

\begin{proof}
Relations \eqref{gcds} imply: 
\ben
\label{components}
\frac{(v_2)}{(v_1,v_2)}=(c)~,~\frac{(v_1)}{(v_1,v_2)}=(b)~,~\frac{(v_2)}{(u_1,v_2)}=\frac{(v_1)}{(u_2,v_1)}=(a')~,~\frac{(u_1)}{(u_1,v_2)}=\frac{(u_2)}{(u_2,v_1)}=(d')~.
\een
These relations show that 
$\epsilon_\0(v_1,v_2)$ and $\epsilon_\1(v_1,v_2; W)$ belong to the equivalence classes $\bepsilon_\0(v_1,v_2)$ and $\bepsilon_\1(v_1,v_2; W)$ 
and that we have $\bepsilon_\1(v_1,v_2; W)=\bepsilon_\1(v_2,v_1; W)$.

For an even morphism $f: e_{1} \to e_{2}$ in $\MF(R,W)$, the first
equation in \eqref{fdelem} shows that the condition
$\fd_{e_1,e_2}(f)=0$ amounts to:
\be
f_{\1\1}c-f_{\0\0}b=0~~.
\ee
Since $b$ and $c$ are coprime, this condition is equivalent with
the existence of an element $\gamma\in R$ such that $f_{\0\0}=\gamma c$
and $f_{\1\1}=\gamma b$. Thus:
\ben
\label{f70}
f=\gamma\left[\begin{array}{cc} c & 0 \\0 & b \end{array}\right]=\gamma \epsilon_\0(v_1,v_2)~~.
\een
On the other hand, the second equation in \eqref{fdelem} shows that an
odd morphism $g: e_{1} \to e_{2}$ in $\MF(R,W)$ satisfies
$\fd_{e_1,e_2}(g)=0$ iff:
\be
a g_{\0\1}+d g_{\1\0}=0~~.
\ee
Since $a'$ and $d'$ are coprime, this condition is equivalent with the
existence of an element $\gamma\in R$ such that $g_{\1\0}=\gamma a'$
and $g_{\1\1}=-\gamma d'$. Thus:
\ben
\label{f71}
g=\gamma\left[\begin{array}{cc} 0 & a' \\-d' & 0 \end{array}\right]=\gamma \epsilon_\1(v_1,v_2; W)~~.
\een
\qed
\end{proof}

\

\begin{Proposition}
\label{prop:HomHMF}
Let $v_i$ be as in Proposition \ref{prop:HomZMF}. Then
$\Hom_{\HMF(R,W)}^\0(e_{1}, e_{2})$ and $\Hom_{\HMF(R,W)}^\1(e_{1},
e_{2})$ are cyclically presented cyclic $R$-modules generated respectively by
the matrices $\epsilon_\0(v_1,v_2)$ and $\epsilon_\1(v_1,v_2;W)$,
whose annihilators are equal to each other and coincide
with the following principal ideal of $R$:
\be
\alpha_W(v_1,v_2)\eqdef \langle v_1,u_1,v_2,u_2\rangle=\langle s\rangle~~.
\ee
\end{Proposition}

\begin{proof}
Let $f\in \Hom_{\ZMF(R,W)}^\0(e_{1}, e_{2})$. Then $f$ is exact
iff there exists an odd morphism $g\in \Hom_{\MF(R,W)}^\1(e_{1},e_{2})$ such that:
\be
f=\fd_{e_1,e_2}(g)=(a g_{\0\1} +g_{\1\0} d) \left[\begin{array}{cc} c & 0 \\0 & b\end{array}\right]~~.
\ee
Comparing this with \eqref{f70}, we find that $f$ is exact if and only
if $s\in (a,d)$ divides $\gamma$.  This implies that the principal ideal
generated by the element:
\be
s\in ((v_1,v_2),(u_1,u_2))= (v_1,u_1,v_2,u_2)
\ee 
is the annihilator of $\Hom_{\ZMF(R,W)}^\0(e_{1},e_{2})$.

On the other hand, an odd morphism $g\in
\Hom_{\ZMF(R,W)}^\1(e_{1},e_{2})$ is exact iff there exists an
even morphism $f\in \Hom_{\MF(R,W)}^\0(e_{1},e_{2})$ such that:
\be
g=\fd_{e_1,e_2}(f)=(f_{\1\1}c-f_{\0\0}b)\left[\begin{array}{cc} 0 & a \\-d & 0 \end{array}\right]~~.
\ee
Comparing with \eqref{f71} and recalling that $(b,c)=(1)$, we find that
$g$ is exact iff $(a,d) | \gamma$. Hence the annihilator of
$\Hom_{\HMF(R,W)}^\1(e_{1},e_{2})$ coincides with that of
$\Hom_{\HMF(R,W)}^\0(e_{1},e_{2})$.  \qed
\end{proof}

\begin{remark}\label{rem:3}
Since $s$ is a critical divisor of $W$, we have $\fI_W
\Hom_{\HEF(R,W)}(e_1,e_2)=0$, where $\fI_W$ denotes the critical ideal
of $W$ defined in \eqref{IW}. In particular, $\HEF(R,W)$ can be viewed as an
$R/\fI_W$-linear category.
\end{remark}

\

\noindent Let $\fDiv(W)\eqdef \{d\in R \, \big| \, d|W\}$ and consider the
function $\alpha_W:\fDiv(W)\times \fDiv(W)\rightarrow G_+(R)$ defined
in Proposition \ref{prop:HomHMF}. This function is symmetric since
$\alpha_W(v_1,v_2)=\alpha_W(v_2,v_1)$. Let $1_{G(R)}=\langle
1\rangle=R$ denote the neutral element of the group of divisibility
$G(R)$, whose group operation we write multiplicatively.

\

\begin{Proposition}
\label{prop:mult}
The symmetric function $\alpha_W(v_1,v_2)$ is multiplicative with
respect to each of its arguments in the following sense:
\begin{itemize}
\itemsep 0.0em
\item For any two relatively prime elements $v_2$ and $\Wv_2$ of $R$
  such that $v_2 \Wv_2$ is a divisor of $W$, we have:
\ben
\label{alpharmult}
\alpha_W(v_1, v_2\Wv_2)=\alpha_W(v_1,v_2) \alpha_W(v_1, \Wv_2)~~
\een
and $\alpha_W(v_1,v_2)+\alpha_W(v_1, \Wv_2)=1_{G(R)}$, where $+$
denotes the sum of ideals of $R$.
\item For any two relatively prime elements $v_1$ and $\Wv_1$ of $R$
  such that $v_1\Wv_1$ is a divisor of $W$, we have:
\ben
\label{alphalmult}
\alpha_W(v_1 \Wv_1, v_2)=\alpha_W(v_1,v_2) \alpha_W(\Wv_1, v_2)~~
\een
and $\alpha_W(v_1,v_2)+\alpha_W(\Wv_1, v_2)=1_{G(R)}$, where $+$
denotes the sum of ideals of $R$.
\end{itemize}
\end{Proposition}

\begin{proof}
To prove the first statement, we start from relation
\eqref{alpha_alt}, which allows us to write:
\ben
\label{alphav2v2prime}
\alpha_W(v_1,v_2\Wv_2)=\Big\langle \frac{(u_1,v_2\Wv_2)}{\left((u_1,v_2\Wv_2),\frac{v_2\Wv_2}{(v_1,v_2\Wv_2)}\right)}\Big \rangle~~,
\een
where $u_1=W/v_1$. Recall that the function $(-,r)$ is multiplicative
on relatively prime elements for any $r\in R^\times$, i.e. $(xy,r)= (x,r)(y,r)$. Thus:
\ben
\label{gcdfact}
(u_1,v_2\Wv_2)=(u_1,v_2)(u_1,\Wv_2)~,~~~(v_1,v_2\Wv_2)=(v_1,v_2)(v_1,\Wv_2)~~.
\een
The second of these relations gives
$\frac{(v_2\Wv_2)}{(v_1,v_2\Wv_2)}=\frac{(v_2)}{(v_1,v_2)}\frac{(\Wv_2)}{(v_1,\Wv_2)}$.
Notice that $(\frac{(v_2)}{(v_1,v_2)},\frac{(\Wv_2)}{(v_1,\Wv_2)})=(1)$ since
$v_2$ and $\Wv_2$ are coprime. Hence:
\beqa
&&\left((u_1,v_2\Wv_2),\frac{(v_2)(\Wv_2)}{(v_1,v_2\Wv_2)}\right)= \left((u_1,v_2\Wv_2),\frac{(v_2)}{(v_1,v_2)}\right)\left((u_1,v_2\Wv_2),\frac{(\Wv_2)}{(v_1,\Wv_2)}\right)= \\
&&=\left((u_1,v_2),\frac{(v_2)}{(v_1,v_2)}\right)\left((u_1,\Wv_2),\frac{(v_2)}{(v_1,v_2)}\right)\left((u_1,v_2),\frac{(\Wv_2)}{(v_1,\Wv_2)}\right)\left((u_1,\Wv_2),\frac{(\Wv_2)}{(v_1,\Wv_2)}\right)~,
\eeqa
where in the last equality we used the first relation in
\eqref{gcdfact} and noticed that $(u_1,v_2)$ and $(u_1,\Wv_2)$ are
coprime (since $(v_2,\Wv_2)=(1)$), which allows us to use similar-multiplicativity of
the function $(-,r)$ for $(r)=\frac{(v_2)}{(v_1,v_2)}$ and for
$(r)= \frac{(\Wv_2)}{(v_1,\Wv_2)}$. Since $(v_2,\Wv_2)=(1)$, we have
$\left((u_1,\Wv_2),\frac{(v_2)}{(v_1,v_2)}\right) = \left((u_1,v_2),\frac{(\Wv_2)}{(v_1,\Wv_2)}\right)=(1)$. Thus:
\be
\left((u_1,v_2\Wv_2),\frac{(v_2)(\Wv_2)}{(v_1,v_2\Wv_2)}\right)= \left((u_1,v_2),\frac{(v_2)}{(v_1,v_2)}\right)\left((u_1,\Wv_2),\frac{(\Wv_2)}{(v_1,\Wv_2)}\right)~~.
\ee
Using this and the first equation of \eqref{gcdfact} in the expression
\eqref{alphav2v2prime} gives relation \eqref{alpharmult}. The second
statement now follows from the first by symmetry of $\alpha_W$.  \qed
\end{proof}

\subsubsection{Isomorphisms in $\HEF(R,W)$}

\

\

\noindent We start with a few lemmas.

\

\begin{lemma}\label{lemma:2Bez}
Let $s,x,y,z$ be four elements of R. Then the equation:
\begin{equation}\label{f92}
s(g_1x+g_2y) +g_3z=1
\end{equation}
has a solution $(g_1,g_2,g_3)\in R^3$ iff $(s(x,y),z)=(1)$. 
\end{lemma}
\begin{proof}
Let $t$ be a gcd of $x$ and $y$. We treat each implication in turn:
\begin{enumerate}[1.]
\itemsep 0.0em
\item Assume that $(g_1, g_2, g_3)\in R^3$ is a solution. Then $t$
  divides $g_1x+g_2y$, so there exists $g_4\in R$ such that
  $g_1x+g_2y=g_4t$. Multiplying both sides with $s$ and using
  \eqref{f92}, this gives $stg_4+g_3z=1$, which implies $(st,z)=(1)$.
\item Assume that $(st,z)=(1)$. Then there exist $g_3,g_4\in R$ such
  that:
\ben
\label{stgz}
stg_4+g_3z=1~~. 
\een
Since $(t)=(x,y)$, the B\'ezout identity shows that there exist
$\widetilde{g}_1, \widetilde{g}_2\in R$ such that
$\widetilde{g}_1x+\widetilde{g}_2y=t$. Substituting this into
\eqref{stgz} shows that $(g_1,g_2,g_3)$ satisfies \eqref{f92}, where
$g_1\eqdef \widetilde{g}_1g_4$ and $g_2\eqdef \widetilde{g}_2g_4$. \qed
\end{enumerate}
\end{proof}

\

\begin{lemma}
\label{lemma:sys}
Let $s, a',b,c,d'$ be five elements of $R$ such that $(a',d')=(1)$. Then
the system of equations:
\begin{equation}
\label{S1}
\begin{cases}
bcg  - s(a'b g_{1}+ cd'g_{2})=1\\
bcg - s(a'c h_{1}+ bd'h_{2})=1
\end{cases}
\end{equation}
has a solution $(g,g_1,g_2,h_1,h_2)\in R^5$ iff $a',b,c,d'$ are pairwise
coprime and $(bc,s)=(1)$.  
\end{lemma}
\begin{proof} Consider the two implications in turn.

\begin{enumerate}[1.]
\itemsep 0.0em
\item Assume that \eqref{S1} has a solution $(g,g_1,g_2,h_1,h_2)\in
  R^5$. By Lemma \ref{lemma:2Bez}, we must have $(bc,s(a'b,cd'))=(1)$
  and $(bc,s(a'c,bd'))=(1)$. This implies $(bc,s)=(1)$ and $(b,c)=(1)$. If a
  prime element $p\in R$ divides $(a'b,cd')$, then it divides both
  $a'b$ and $cd'$, hence $p|c$ or $p|b$ since $(a',d')=(1)$. Thus
  $p|bc$, which contradicts the fact that that $bc$ and $s(a'b,cd')$
  are coprime. It follows that we must have $(a'b,cd')=(1)$. Similarly,
  the second equation implies that we must have $(a'c,bd')=(1)$. Since
  $(a',d')=(1)$ and $(b,c)=(1)$, the last two conditions imply that
  $a',b,c,d'$ must be pairwise coprime.

\item Conversely, assume that $a',b,c,d'$ are pairwise coprime and
  $(bc,s)=(1)$. Following the strategy and notations of the previous
  lemma, we first solve the equation $bcg-sg_4=1$ for $g$ and $g_4$
  using the B\'ezout identity. Using the same identity, we solve the
  system:
\begin{equation}
\begin{cases}
a'b \widetilde{g}_{1}+ cd'\widetilde{g}_{2}=1\\ 
a'c \widetilde{h}_{1}+ bd'\widetilde{h}_{2}=1
\end{cases} \ ,
\end{equation}
obtaining the solution $(g, g_4\widetilde{g}_{1},g_4 \widetilde{g}_{2},g_4
\widetilde{h}_{1},g_4\widetilde{h}_{2})$ of \eqref{S1}. \qed
\end{enumerate}
\end{proof}

\

\begin{Proposition}
\label{prop:EFcrit1}
With the notations \eqref{param}, we have:
\begin{enumerate}
\itemsep 0.0em
\item $e_1$ and $e_2$ are isomorphic in $\hef(R,W)$ iff
  $a',b,c,d'$ are pairwise coprime and $(bc,s)=(1)$.
\item An odd isomorphism between $e_1$ and $e_2$ in $\HEF(R,W)$ exists 
  iff $a',b,c,d'$ are pairwise coprime and $(a'd',s)=(1)$.
\end{enumerate}
\end{Proposition}

\begin{proof}
\begin{enumerate}
\itemsep 0.0em
\item Proposition \ref{prop:HomZMF} gives:
\be
\Hom_{\zef(R,W)}(e_1,e_2)=R\left[\begin{array}{cc} c & 0\\ 0&
    b \end{array}\right]~\mathrm{and}~~\Hom_{\zef(R,W)}(e_2,e_1)=R\left[\begin{array}{cc} b & 0\\ 0&
    c \end{array}\right]~~.  
\ee
Two non-zero morphisms $f_{12}=\alpha \left[\begin{array}{cc} c &
    0\\ 0& b \end{array}\right]\in \Hom_{\zef(R,W)}(e_1,e_2)$ and
$f_{21}=\beta\left[\begin{array}{cc} b & 0\\ 0&
    c \end{array}\right]\in \Hom_{\zef(R,W)}(e_2,e_1)$ (where
$\alpha,\beta\in R^\times$) induce mutually inverse
isomorphisms in $\hef(R,W)$ iff:
\be
f_{21}f_{12}=1+\fd_{e_1,e_1}(g)~~,~~f_{12}f_{21}=1+\fd_{e_2,e_2}(h)~~
\ee
for some $g,h\in \End^\1_R(R^{1|1})$. These conditions read:
\beqa
&& \alpha\beta\left[\begin{array}{cc} bc & 0\\ 0& bc \end{array}\right]=\left[\begin{array}{cc} 1 & 0\\ 0& 1 \end{array}\right]+ 
(ab g_{\0\1}+g_{\1\0} cd)\left[\begin{array}{cc} 1 & 0\\ 0& 1 \end{array}\right]=\left[\begin{array}{cc} 1 & 0\\ 0& 1 \end{array}\right]+ 
s(a'b g_{\0\1}+g_{\1\0} cd')\left[\begin{array}{cc} 1 & 0\\ 0& 1 \end{array}\right]~~\\
&& \alpha\beta\left[\begin{array}{cc} bc & 0\\ 0& bc \end{array}\right] = \left[\begin{array}{cc} 1 & 0\\ 0& 1 \end{array}\right]+ 
(ac h_{\0\1}+h_{\1\0} bd)\left[\begin{array}{cc} 1 & 0\\ 0& 1 \end{array}\right]=\left[\begin{array}{cc} 1 & 0\\ 0& 1 \end{array}\right]+ 
s(a'c h_{\0\1}+h_{\1\0} bd')\left[\begin{array}{cc} 1 & 0\\ 0& 1 \end{array}\right]~~
\eeqa
and hence amount to the following system of equations for $\alpha\beta$ and $g,h$: 
\begin{equation}
\begin{cases}
\alpha\beta bc-s(a'b g_{\0\1}+g_{\1\0} cd')=1\\
\alpha\beta bc-s(a'c h_{\0\1}+h_{\1\0} bd')=1~~.
\end{cases}
\end{equation}
Since this system has the form \eqref{S1}, Lemma \ref{lemma:sys} shows
that it has solutions iff $a',b,c,d'$ are pairwise coprime and
$(bc,s)=(1)$.

\item Proposition \ref{prop:HomZMF} gives
  $\Hom^\1_{\ZEF(R,W)}(e_1,e_2)=\Hom^\1_{\ZEF(R,W)}(e_2,e_1)=R
  \left[\begin{array}{cc} 0 & a'\\ -d'& 0 \end{array}\right]$.  Two
  non-zero odd morphisms $g_{12}=\alpha\left[\begin{array}{cc} 0 &
      a'\\ -d'& 0 \end{array}\right]\in \Hom^\1_{\ZEF(R,W)}(e_1,e_2)$
  and $g_{21}=\beta\left[\begin{array}{cc} 0 & a'\\ -d'&
      0 \end{array}\right]\in \Hom^\1_{\ZEF(R,W)}(e_2,e_1)$ (with $\alpha,\beta\in R^\times$) 
induce mutually inverse isomorphisms in $\HEF(R,W)$ iff:
\be
g_{21}g_{12}=1+\fd_{e_1,e_1}(f)~~,~~g_{12}g_{21}=1+\fd_{e_2,e_2}(q)~~
\ee
for some $f,q\in \End^\0_R(R^{1|1})$. This gives the equations:
\beqa
&& \alpha\beta\left[\begin{array}{cc} a'd' & 0\\ 0& a'd' \end{array}\right] =\left[\begin{array}{cc} 1 & 0\\ 0& 1 \end{array}\right]+ 
(ac f_{\0\0}+f_{\1\1} bd) \left[\begin{array}{cc} 1 & 0\\ 0& 1 \end{array}\right]~~,\\
&& \alpha\beta\left[\begin{array}{cc} a'd' & 0\\ 0& a'd' \end{array}\right] =\left[\begin{array}{cc} 1 & 0\\ 0& 1 \end{array}\right]+ 
(ab q_{\0\0}+q_{\1\1} cd) \left[\begin{array}{cc} 1 & 0\\ 0& 1
\end{array}\right]~~\\
\eeqa
which amount to the system:
\begin{equation}
\begin{cases}
\alpha\beta a'd'-s(a'c f_{\0\0}+f_{\1\1} bd')=1\\
\alpha\beta a'd'-s(a'b q_{\0\0}+q_{\1\1} cd')=1~~.
\end{cases}
\end{equation}
This system again has the form \eqref{S1}, as can be seen by
the substitution of the quadruples $(b,c,a',d'):=(a',d',b,c)$.  As a consequence, it has a solution
iff $a',b,c,d'$ are pairwise coprime and $(a'd',s)=(1)$. 
 \qed
\end{enumerate}
\end{proof}

\

\begin{Corollary} 
\label{prop:similar}
Similar elementary matrix factorizations of $W$ are isomorphic in $\hef(R,W)$.
\end{Corollary}

\begin{proof}
The statement follows immediately from Proposition \ref{prop:EFcrit1}
by taking $a=v\gamma$, $b=\gamma^{-1}$, $c=1$, $d=u$ (where $\gamma\in
U(R)$), since the gcd in $R$ is defined modulo $U(R)$. \qed
\end{proof}

\

\begin{Proposition}\label{prop:suspiso}
Any elementary matrix factorization of $W$ is odd-isomorphic in
$\HEF(R,W)$ to its suspension:
\be
e_v\simeq_{\HEF(R,W)} \Sigma e_v=e_{-u}~~,
\ee
where $u=W/v$. 
\end{Proposition}

\begin{proof}
Let $s\in (u,v)$. The isomorphism follows from Proposition
\ref{prop:EFcrit1} for $a'=1=d'$, $b=-v/s$, $c=-u/s$:
\be
\left[\begin{array}{cc} 0 &
    v\\ u & 0\end{array}\right]\simeq\left[\begin{array}{cc} 0 &
    bs\\ cs & 0\end{array}\right]\simeq\left[\begin{array}{cc} 0 &
    cs\\ bs & 0\end{array}\right]\simeq \left[\begin{array}{cc} 0 &
    u\\ v & 0\end{array}\right] \ ,
\ee
since $(ad,s)=(1)$. 
\qed
\end{proof}

\begin{remark}
An odd isomorphism in $\HEF(R,W)$ between $e_v$ and $\Sigma
e_v=e_{-u}$ can also be obtained more abstractly by transporting the
identity endomorphism of $e_v$ through the isomorphism of $R$-modules
$\Hom^\1(e_v,e_{-u})=\Hom^\1(e_v,\Sigma e_v)\simeq
\Hom_{\hmf(R,W)}(e_v,e_v)$ which results by taking $v_1=v$ and
$v_2=-u$ in the first line of \eqref{graded}. Since $e_{-u}$ is
similar to $e_u$, the Proposition implies that $e_v$ and $e_u$ are
oddly isomorphic. When $v=1$, both $e_v=e_1$ and $e_u=e_W$ are zero
objects and we have
$\Hom_{\HMF(R,W)}^\0(e_1,e_W)=\Hom_{\HMF(R,W)}^\1(e_1,e_W)=\{0\}$, so
the odd isomorphism is the zero morphism.
\end{remark}

\

\begin{Proposition}
\label{prop:isomW1}
Let $W=W_1W_2$ with $(W_1,W_2)=(1)$ and let $v$ be a divisor of
$W_1$. Then $e_{W_2v}\simeq_{\hmf(R,W)} e_v$.
\end{Proposition}

\begin{proof}
Let $u_0\eqdef \frac{W_1}{v}$. Setting $v_1= v$,
$u_1=\frac{W}{v}=W_2u_0$, $v_2=W_2 v$ and $u_2=\frac{W}{v_2}=u_0$, we
compute:
\beqa
&&a\in (v_1,v_2)=(v)~,~b=\frac{v_1}{a}\in (1)~,~c=\frac{v_2}{a}\in (W_2)~,~d\in \frac{W}{[v_1,v_2]}=\left(\frac{W_1}{v}\right)~~\\
&&s\in (u_1,v_1,u_2,v_2)=(u_0,v)~,~a'=a/s\in \frac{v}{(u_0,v)}~~,~~d'=d/s\in \frac{(W_1)}{(v)(u_0,v)}=\frac{(u_0)}{(u_0,v)}~~.
\eeqa 
It is clear that $a',b,c,d'$ are mutually coprime and that $(s,bc)=(1)$. \qed
\end{proof}

\subsubsection{The composition of morphisms in $\HEF(R,W)$}

\

\

\begin{Proposition}
\label{prop:epsilon_comp}
Given three divisors $v_1$, $v_2$ and $v_3$ of $W$, we have the
following relations:
\beqan
\label{ecomp}
\bepsilon_\0(v_2,v_3)\bepsilon_\0(v_1,v_2)&=&\frac{(v_2) (v_1,v_3)}{(v_1,v_2)(v_2,v_3)}\bepsilon_\0(v_1,v_3)\nn\\
\bepsilon_\0(v_2,v_3)\bepsilon_\1(v_1,v_2; W)&= &\frac{(v_2) (u_1,v_3)}{(u_1,v_2)(v_2,v_3)} \bepsilon_\1(v_1,v_3; W)\nn\\
\bepsilon_\1(v_2,v_3; W)\bepsilon_\0(v_1,v_2)&= &\frac{(v_3) (v_1,u_3)}{(v_1,v_2)(u_2,v_3)}\bepsilon_\1(v_3,v_1; W)\\
\bepsilon_\1(v_2,v_3; W)\bepsilon_\1(v_1,v_2; W)&= &-\frac{(u_1) (v_1,v_3)}{(u_2,v_3)(u_1,v_2)}\bepsilon_\0(v_1,v_3)~~.\nn
\eeqan
\end{Proposition}
\begin{proof}
\noindent Given three divisors $v_1$, $v_2$ and $v_3$ of $W$, we have: 
\be
\bepsilon_\0(v_2,v_3)\bepsilon_\0(v_1,v_2)= \left[\begin{array}{cc} 
\frac{v_2v_3}{(v_1,v_2)(v_2,v_3)} & 0 \\ 0 & \frac{v_1v_2}{(v_1,v_2)(v_2,v_3)} 
\end{array}\right]= \frac{v_2 (v_1,v_3)}{(v_1,v_2)(v_2,v_3)}\bepsilon_\0(v_1,v_3)= \frac{[v_1,v_2,v_3] 
(v_1,v_3)}{(v_1,v_2,v_3)[v_1,v_3]}\bepsilon_\0(v_1,v_3)
\ee
where we used the identity: 
\ben
\label{3id}
[a,b,c](a,b)(b,c)(c,a)=(a)(b)(c)\, (a,b,c)~~.
\een
This establishes the first of equations \eqref{ecomp}. The remaining equations follow similarly \qed.
\end{proof}

\

\begin{Corollary}
\label{cor:EndZMF}
Let $v$ be a divisor of $W$ and $u=W/v$. Then:
\begin{enumerate}
\itemsep 0.0em
\item The $R$-algebra $\End_{\zmf(R,W)}(e_v)$ is isomorphic with $R$.
\item We have an isomorphism of $\Z_2$-graded $R$-algebras:
\be
\End_{\ZMF(R,W)}(e_v)\simeq \frac{R[\omega]}{\langle u^2+t\rangle}~~,
\ee
where $\omega$ is an odd generator and $t\in \frac{[u,v]}{(u,v)}$. In particular,
$\End_{\ZMF(R,W)}(e_v)$ is a commutative $\Z_2$-graded ring.
\end{enumerate}
\end{Corollary}

\begin{proof}
For $v_1=v_2=v$, we have $\alpha_W(v,v)=\langle u,v\rangle$. Proposition \ref{prop:HomZMF} gives: 
\be
\bepsilon_\0(v,v)\eqdef \left[\begin{array}{cc} 
1 & 0 \\ 0 & 1\end{array}\right]~~,~~\bepsilon_\1(v,v)\eqdef \left[\begin{array}{cc} 0 & \frac{(v)}{(u,v)} \\-\frac{(u)}{(u,v)} & 0\end{array}\right]
\ee
and we have:
\beqa
&&\bepsilon_\0(v,v)^2=\bepsilon_\0(v,v)\\
&&\bepsilon_\0(v,v)\bepsilon_\1(v,v; W)= \bepsilon_\1(v,v; W)\bepsilon_\0(v,v)=\bepsilon_\1(v,v)\\
&&\bepsilon_\1(v,v; W)^2= -\frac{[u,v]}{(u,v)}\bepsilon_\0(v,v)~~,
\eeqa
which also follows from Proposition \ref{prop:epsilon_comp}.  Setting
$\omega=\epsilon_\1(v,v; W)$, these relations imply the
desired statements upon using Proposition
\ref{prop:HomZMF}. \qed
\end{proof}

\

\begin{Corollary}
\label{cor:EndHMF}
Let $v$ be a divisor of $W$ and $u=W/v$. Then:
\begin{enumerate}
\itemsep 0.0em
\item The $R$-algebra $\End_{\hmf(R,W)}(e_v)$ is isomorphic with
$R/\langle d\rangle=R/\langle u, v \rangle$, where $d\in (u,v)$.
\item We have an isomorphism of $\Z_2$-graded $R$-algebras:
\be
\End_{\HMF(R,W)}(e_v)\simeq \frac{\left(R/\langle d\rangle \right)[\omega]}{\langle u^2+t\rangle}~~,
\ee
where $\omega$ is an odd generator, $d\in (u,v)$ and $t\in \frac{[u,v]}{(u,v)}$. In particular,
$\End_{\ZMF(R,W)}(e_v)$ is a supercommutative $\Z_2$-graded ring.
\end{enumerate}
\end{Corollary}
\begin{proof}
The same relations as in the previous Corollary imply the conclusion
upon using Proposition \ref{prop:HomHMF}. \qed
\end{proof}

\

\begin{Corollary}
\label{cor:zero}
An elementary matrix factorization $e_v$ is a zero object of
$\hmf(R,W)$ iff $(u,v)=(1)$, where $u=W/v$.
\end{Corollary}

\begin{proof}
The $R$-algebra $\End_{\HMF(R,W)}^\0(e_v)\simeq R/t$
(where $u=W/v$) vanishes iff $(u,v)=(1)$. \qed
\end{proof}

\subsection{Localizations}

Let $S\subset R$ be a multiplicative subset of $R$ containing the
identity $1 \in R$ and $\lambda_S:R \to R_S$ denote the natural ring
morphism from $R$ to the localization $R_S=S^{-1}R$ of $R$ at $S$. For
any $r\in R$, let $r_S\eqdef \lambda_S(r)=\frac{r}{1}\in R_S$ denote
its extension. For any $R$-module $N$, let $N_S=S^{-1}N=N\otimes_R
R_S$ denote the localization of $N$ at $S$. For any morphism of
$R$-modules $f:N \rightarrow N' $, let $f_S\eqdef f\otimes_R
\id_{R_S}:N_S\rightarrow N'_S$ denote the localization of $f$ at
$S$. For any $\Z_2$-graded $R$-module $M=M^\0\oplus M^\1$, we have
$M_S=M^\0_S\oplus M^\1_S$, since the localization functor is exact. In
particular, localization at $S$ induces a functor from the category of
$\Z_2$-graded $R$-modules to the category of $\Z_2$-graded
$R_S$-modules.

Let $a=(M,D)$ be a matrix factorization of $W$. The {\em localization
 of $a$ at $S$} is the following matrix factorization of $W_S$ over
the ring $R_S$:
\be
a_S\eqdef (M_S,D_S)\in \MF(R_S,W_S)~~.
\ee
It is clear that this extends
to an even dg functor $\loc_S:\MF(R,W)\rightarrow \MF(R_S,W_S)$, which
is $R$-linear and preserves direct sums. In turn, this induces dg
functors $\ZMF(R,W)\rightarrow \ZMF(R_S,W_S)$, $\BMF(R,W)\rightarrow
\BMF(R_S,W_S)$, $\HMF(R,W)\rightarrow \HMF(R_S,W_S)$ and
$\hmf(R,W)\rightarrow \hmf(R_S,W_S)$, which we again denote by
$\loc_S$.  We have $\loc_S(a)=a_S$ for any matrix factorization $a$ of
$W$ over $R$. 

\

\begin{Proposition}
\label{prop:loc_triang}
The functor $\loc_S:\hmf(R,W)\rightarrow \hmf(R_S,W_S)$ is a triangulated functor.
Moreover, the strictly full subcategory of $\hmf(R,W)$ defined through: 
\be
K_S\eqdef \big\{a\in \Ob[\hmf(R,W)] ~\big|~  a_S\simeq_{\hmf(R_S,W_S)} 0\big\}
\ee
is a triangulated subcategory of $\hmf(R,W)$.
\end{Proposition}

\begin{proof}
It is clear that $\loc_S$ commutes with the cone construction (see
\cite{Langfeldt} for a detailed account of the latter). It is also
clear that the subcategory $K_S$ is closed under shifts. Since any
distinguished triangle in which two objects vanish has the property
that its third object also vanishes, $K_S$ is also closed under
forming triangles.
\qed
\end{proof}

\

\begin{Proposition}
\label{prop:locHom}
For any matrix factorizations $a,b$ of $W$, there exists a natural
isomorphism of $\Z_2$-graded $R_S$-modules:
\be
\Hom_{\HMF(R_S,W_S)}(a_S,b_S)\simeq_{R_S} \Hom_{\HMF(R,W)}(a,b)_S ~~.
\ee
\end{Proposition}

\begin{proof}
Follows immediately from the fact that localization at $S$ is an exact functor 
from $\Mod_R$ to $\Mod_{R_S}$. \qed
\end{proof}

\subsection{Behavior of $\hef(R,W)$ under localization}

\

\

\begin{Lemma}
\label{lemma:quot}
The following statements are equivalent for any elements $s,r$ of $R$: 
\begin{enumerate}
\itemsep 0.0em
\item $(s,r)=(1)$
\item The class of $s$ modulo the ideal $\langle r\rangle$ is a unit of the ring $R/\langle r\rangle$. 
\end{enumerate}
\end{Lemma}

\begin{proof}
We have $(s,r)=(1)$ iff there exist elements $a,b\in R$ such that
$as+br=1$. In turn, this is equivalent with the condition $\bar{a}\bar{s}=\bar{1}$ in
the ring $R/\langle r\rangle$, where $\bar{x}=x+\langle r\rangle$ denotes the equivalence class of
an element $x\in R$ modulo the ideal $\langle r\rangle$. \qed
\end{proof}

\

\noindent Consider the multiplicative set:  
\be
S_W\eqdef \{s\in R ~\big|~ (s,W)=(1)\}~~.
\ee
Since $0\not \in S_W$, the localization $R_S=S^{-1}R$ of $R$ at any
multiplicative set $S\subset S_W$ is a sub-ring of the field of
fractions $K$ of $R$:
\be
R_S=\{\frac{r}{s} ~\big|~\, r\in R, s\in S\}\subset K~~.
\ee
In particular, $R_S$ is an integral domain. 

\

\begin{Proposition}
\label{prop:eloc}
Let $S$ be any multiplicative subset of $R$ such that $S\subset
S_W$. Then the localization functor $\loc_S:\hmf(R,W)\rightarrow
\hmf(R_S,W_S)$ restricts to an $R$-linear equivalence of categories
between $\hef(R,W)$ and $\hef(R_S,W_S)$.
\end{Proposition}

\begin{proof}
Since $\loc_S$ preserves the reduced rank of matrix factorizations, it
is clear that it restricts to a functor from $\hef(R,W)$ to
$\hef(R_S,W_S)$.  Given two elementary factorizations
$e_{v_1},e_{v_2}\in \Ob[\hef(R,W)]$, let $r\in 
(v_1,v_2,W/v_1,W/v_2)$. By Proposition \ref{prop:locHom}, we have:
\ben
\label{loceq}
\Hom_{\hef(R_S,W_S)}((e_{v_1})_S,(e_{v_2})_S)\simeq_R  \Hom_{\hef(R,W)}(e_{v_1},e_{v_2})_S~~. 
\een 
Let $s$ be any element of $S$. Since $S$ is a subset of $S_W$, we have
$(s,W)=(1)$ and hence $(s,r)=(1)$ since $r$ is a divisor of $W$. By
Lemma \ref{lemma:quot}, the image $\bar{s}=s+\langle r\rangle$ is a
unit of the quotient ring $R/\langle r\rangle$, hence the operator of
multiplication with $s$ is an isomorphism of the cyclic $R$-module
$\Hom_{\HMF(R,W)}(e_{v_1},e_{v_2})\simeq R/\langle r\rangle$. Thus
every element of $S$ acts as an automorphism of this module, which
implies that the localization map
$\Hom_{\hmf(R,W)}(e_{v_1},e_{v_2})\rightarrow
\Hom_{\hmf(R,W)}(e_{v_1},e_{v_2})_S$ is an isomorphism of $R$-modules
(where $\Hom_{\hmf(R,W)}(e_{v_1},e_{v_2})_S$ is viewed as an
$R$-module by the extension of scalars $R\rightarrow R_S$). Combining
this with \eqref{loceq} shows that the restriction
$\loc_S:\hef(R,W)\rightarrow \hef(R_S,W_S)$ is a full and faithful
functor.

Now let $e_x$ be an elementary factorization of $W_S$ corresponding to
the divisor $x$ of $W_S=W/1$ in the ring $R_S$. Let $y=W_S/x\in R_S$.
Write $x=v/s$ and $y=u/t$ with $x,y\in R$ and $s,t\in S$ chosen such
that $(v,s)= (u,t)=(1)$.  Then the relation $xy=W_S$ amounts to $uv=st
W$. Since $S$ is a subset of $S_W$, we have $(s,W)=(t,W)=(1)$.  Thus
$st|uv$, which implies $s|v$ and $t|u$ since $(v,s)= (u,t)=(1)$. Thus
$v=v_1 t$ and $u=u_1 s$ with $u_1,v_1\in R$ and we have $u_1v_1=W$.
This gives $x=\gamma v_1$ and $y=\gamma^{-1} u_1$, where $\gamma\eqdef
t/s$ is a unit of $R_S$. It follows that $e_x$ is similar to the
elementary matrix factorization $e_{v_1}$ of $W_S$ over $R_S$, and
hence isomorphic to the latter in the category $\hef(R_S,W_S)$ by
Proposition \ref{prop:similar}. Since $u_1$ and $v_1$ are divisors of
$W$ satisfying $u_1v_1=W$, we can view $e_{v_1}$ as an elementary
factorization of $W$ over $R$ (it lies in the image of the functor
$\loc_S$). This shows that any objects of $\hef(R_S,W_S)$ is
even-isomorphic with an object lying in the image of the restricted
localization functor, hence the latter is essentially surjective.
\qed
\end{proof}

\subsection{Behavior of $\HEF(R,W)$ under multiplicative partition of $W$}

For any divisor $W_1$ of $W$, let $\HEF_{W_1}(R,W)$ denote the full
subcategory of $\HEF(R,W)$ whose objects are those elementary
factorization $e_v$ of $W$ for which $v$ is a divisor of $W_1$.

\
\begin{Proposition}
\label{prop:EFcrit2}
Let $e_1$ and $e_2$ be as above. Consider elements of $R$ chosen as follows: 
\beqan\label{sxydata}
&& s_1\in (u_1,v_1)= (s)(a',c)(b,d')~,~~s_2\in (u_2,v_2)= (s) (a',b)(c,d')\nn \ ,\nn \\
&& u'_1\eqdef u_1/s=cd'~,~u'_2\eqdef u_2/s=bd'~,~v'_1\eqdef v_1/s=a'b~,~v'_2\eqdef v_2/s=a'c \ ,\nn \\
&& ~x(e_1)\in (s,v'_1)=(s,a'b)~,~~
y(e_1)\in (s,u'_1)=(s,cd')~,\\
&&x(e_2)\in  (s,v'_2)= (s,a'c)~,~~
y(e_2)\in (s,u'_2)=(s,bd')~~.\nn
\eeqan
Then: 
\begin{enumerate}
\itemsep 0.0em
\item $e_1$ and $e_2$ are isomorphic in $\hef(R,W)$ iff:
\begin{equation}
\label{hefiso}
\begin{aligned}
(i)~~& (s_1)=(s_2) \text{~~and}\\
(ii)~~&  \big((x(e_1)), (y(e_1))\big)=\big((x(e_2)),(y(e_2))\big) \text{~~as ordered pairs of elements in } R^\times/U(R).
\end{aligned}
\end{equation}
\item $e_1$ and $e_2$ are isomorphic in $\HEF(R,W)$ iff:
\begin{equation}
\label{f30}
\begin{aligned}
(i)~~& (s_1)=(s_2) \text{~~and}\\
(ii)~~&  \big\{(x(e_1)), (y(e_1)) \big\} =  \big\{ (x(e_2)), (y(e_2))\big\} \text{~~as unordered pairs of elements in } R^\times/U(R).
\end{aligned}
\end{equation}
\end{enumerate}
Notice that $(s_1)=(s_2)$ implies $(s_1)=(s)=(s_2)$, with $s$  defined in \eqref{ds}.
\end{Proposition}
\begin{proof}

\

\begin{enumerate}
\item Assume that $e_1\simeq_{\hmf(R,W)}e_2$. By Proposition
  \ref{prop:EFcrit1}, part 1, to such pair of elementary
  factorizations we can associate four pairwise coprime divisors
  $a',b,c,d'$ of $W$ such that $v_1=a'bs$, $u_1=d'cs$, $v_2=a'cs$,
  $u_2=d'bs$ and together with the equality $(bc,s)=(1)$. Thus
  $(s_1)=(v_1,u_1)=(a'bs,d'cs)=(s)$, since $a'b$ and $d'c$ are
  coprime. Similarly, $(s_2)=(s)$. The equality $(bc,s)=(1)$ is equivalent
  to $(b,s)=(1)$ and $(c,s)=(1)$. Using this, we compute:
\ben 
(x(e_1))= (s,v'_1)=(s,a'b)=(s,a')=(s,a'c)=(s,v'_2)=(x(e_2))
\een
Acting similarly, we also find $(y(e_1))=(s,d')=(y(e_2))$. Thus \eqref{hefiso} holds.

Now assume that \eqref{hefiso} is satisfied for two elementary factorizations
$e_1$ and $e_2$.  Let $s\in (s_1)=(s_2)$ and define $a',b,c,d'$ as before,
following (\ref{param}). By the very construction, $(b,c)=(1)$ and
$(a',d')=(1)$. We first show that $s_1\sim s_2$ all $a',b,c,d'$ are pairwise
coprime. Indeed, if we assume that $p|(a',b)$ then $s_1\sim s_2$ implies:
\be
(s)(a',c)(b,d')= (s_1)=(s_2)=(s)(a',b)(c,d')~~.
\ee
Since $p$ divides the right hand side, it should divide $(a',c)(b,d')$
and $(c,b)=(1)$ and $(a',d')=(1)$. Thus $p\in U(R)$. It much the same way
we prove that other pairs from $a',b,c,d'$ are coprime.

\noindent Condition (ii) in \eqref{hefiso} reads:
\be
(s,a'b)=(x(e_1))=(x(e_2))=(s,a'c) ~~ .
\ee
If $p|b$ and $p|s$ then $p|(s,b)$ and thus $p|(s,a'b)$. By the
equality above, we also have $p|(s,a'c)$ and hence $p|a'c$. But $b$ is
coprime with both $a'$ and $c$, thus $p\in U(R)$. Similarly, $p|c$ and
$p|s$ implies $p\in U(R)$. Thus $(bc,s)=(1)$. Note that $(y(e_1))=(y(e_2))$
is now automatically satisfied. Proposition \ref{prop:EFcrit1}, part 1
implies that $e_1\simeq_{\hef(R,W)} e_2$.

\item Assume $e_1\simeq_{\HEF(R,W)} e_2$. If the isomorphism is even,
  then it comes from the isomorphism in $\hef(R,W)$ and part 1 above
  already proves that \eqref{hefiso} and thus also \eqref{f30}. Thus we
  can assume that the isomorphism is odd. We will prove that $(s_1)= (s_2)$
  and $(x(e_1))=(y(e_2))$, $(x(e_2))=(y(e_1))$.  Applying Proposition
  \ref{prop:EFcrit1}, part 2, we obtain $a',b,c,d'$ pairwise coprime
  and $s$ such that $(s,a'd')=(1)$. Then $(s_1)= (s)=(s_2)$ similarly to part 1
  above. Using $(s,a'd')=(1)$, we also compute:
\be
(x(e_1))=(s,v'_1)= (s,a'b)=(s,b)= (s,d'b)=(s,u'_2)=(y(e_2))
\ee
and also $(x(e_2))=(y(e_1))$. Thus \eqref{f30}.

Assume now that \eqref{f30} is satisfied. Since the statement for even
morphisms is covered by \eqref{hefiso}, we only need to consider the
situation $(x(e_1))=(y(e_2))$ and $(x(e_2))=(y(e_1))$. As in part 1,
$(s_1)= (s_2)$ implies that $a',b,c,d'$ are pairwise
coprime. Condition (ii) reads:
\be
(s,a'b)=(x(e_1))=(y(e_2))=(s,d'b) ~~.
\ee
If we assume that $p|a'$ and $p|s$ then $p|(s,a')$ and
$p|(s,a'b)$. The equality implies $(p|d'b)$. Since $a'$ is coprime with
both $d'$ and $b$, we obtain $p\in U(R)$. Similarly $(d',s)=(1)$ and
thus $(a'd',s)=(1)$. Proposition \ref{prop:EFcrit1}, part 2 implies that
$e_1\simeq_{\HEF(E,W)} e_2$ by an odd isomorphism. \qed
\end{enumerate}
\end{proof}

\

\begin{Proposition}
\label{prop:fact}
Let $W_1$ and $W_2$ be divisors of $W$ such that $W=W_1W_2$ and $(W_1,W_2)=(1)$. 
Then there exist equivalences of $R$-linear $\Z_2$-graded categories: 
\be
\HEF(R,W_1)\simeq \HEF_{W_1}(R,W)~~,~~\HEF(R,W_2)\simeq \HEF_{W_2}(R,W)~~.
\ee
which are bijective on objects. 
\end{Proposition}

\begin{proof}
For any divisor $v$ of $W$, let $e'_v=(R^{1|1},D'_v)$ and
  $e_v=(R^{1|1},D_v)$ be the corresponding elementary factorizations
  of $W_1$ and $W$, where:
\be
D'_v=\left[\begin{array}{cc} 0 & v\\ W_1/v & 0 \end{array}\right]~~,~~D_v=\left[\begin{array}{cc} 0 & v\\ W/v & 0 \end{array}\right]~~.
\ee
For any two divisors $v_1,v_2$ of $W_1$ and any $\kappa\in \Z_2$, we
have $W/v_i=W_2 \frac{W_1}{v_i}$ and $(v_i, W_2)=(1)$. Thus
$(v_1,v_2,W_1/v_1,W_1/v_2)= (v_1,v_2,W/v_1,W/v_2)$. By Proposition
\ref{prop:HomHMF}, this gives:
\be
\Ann(\Hom^\kappa_{\HEF(R,W_1)}(e'_{v_1},e'_{v_2}))=\Ann(\Hom^\kappa_{\HEF(R,W)}(e_{v_1},e_{v_2}))~,~\forall \kappa\in \Z_2~~.
\ee
On the other hand, the modules
$\Hom^\0_{\HEF(R,W_1)}(e'_{v_1},e'_{v_2})$ and
$\Hom^\0_{\HEF(R,W_1)}(e_{v_1},e_{v_2})$ are generated by the same
element $\epsilon_\0(v_1,v_2)$ while 
$\Hom^\1_{\HEF(R,W_1)}(e'_{v_1},e'_{v_2})$ and
$\Hom^\1_{\HEF(R,W_1)}(e_{v_1},e_{v_2})$ are generated by the elements
$\epsilon_\1(v_1,v_2;W_1)$ and $\epsilon_\1(v_1,v_2; W)$,
respectively.  Hence the functor which maps $e'_v$ to $e_v$ for any
divisor $v$ of $W_1$ and takes $\epsilon_\0(v_1,v_2)$ to
$\epsilon_\0(v_1,v_2)$ and $\epsilon_\0(v_1,v_2;W_1)$ to
$\epsilon_\1(v_1,v_2; W)$ for any two divisors $v_1,v_2$ of $W$ is an
$R$-linear equivalence from $\HEF(R,W_1)$ to $\HEF_{W_1}(R,W)$. A
similar argument establishes the equivalence $\HEF(R,W_2)\simeq
\HEF_{W_2}(R,W)$.\qed
\end{proof}

\subsection{Primary matrix factorizations}

\noindent Recall that an element of $R$ is called primary if it is a power of a prime element. 

\

\begin{Definition}
An elementary factorization $e_v$ of $W$ is called {\em primary} if
$v$ is a primary divisor of $W$.
\end{Definition}

\

\noindent Let $\HEF_0(R,W)$ denote the full subcategory of $\HEF(R,W)$
whose objects are the primary factorizations of $W$.

\

\begin{Proposition}
\label{prop:coproduct}
Let $W=W_1W_2$ be a factorization of $W$, where $W_1$ and $W_2$ are
coprime elements of $R$.  Then there exists an equivalence of
$R$-linear $\Z_2$-graded categories: \be \HEF_0(R,W)\simeq
\HEF_0(R,W_1) \vee \HEF_0(R,W_2)~~, \ee where $\vee$ denotes the
coproduct of $\Mod_R$-enriched categories.
\end{Proposition}

\begin{proof}
Let $\HEF_{0,W_i}(R,W)$ denote the full subcategory of $\HEF_0(R,W)$
whose objects are the primary factorizations $e_v$ of $W$ for which
$v$ is a (primary) divisor of $W_i$. Since $W=W_1W_2$ and
$(W_1,W_2)=(1)$, a primary element $v\in R$ is a divisor of $W$ iff it
is either a divisor of $W_1$ or a divisor of $W_2$. Hence
$\Ob\HEF_0(R,W)=\Ob \HEF_{0,W_1}(R,W)\sqcup \Ob \HEF_{0,
  W_2}(R,W)$. For any primary divisors $v_1$ and $v_2$ of $W$ and any
$\kappa\in \Z_2$, we have:
\be
\Hom^\kappa_{\HEF(R,W)}(e_{v_1},e_{v_2})\!\simeq\! R/\langle d \rangle \!\simeq \!\threepartdef{\Hom^\kappa_{\HEF(R,W_1)}(e_{v_1},e_{v_2})}{v_1|W_1~\&~v_2|W_1}
{\Hom^\kappa_{\HEF(R,W_2)}(e_{v_1},e_{v_2})}{v_1|W_2~\&~v_2|W_2}{0}{v_1|W_2~\&~v_2|W_1}~~,
\ee
where $d\in (v_1,v_2,W/v_1,W/v_2)$ and in the third case we used the
fact that $v_1|W_2$ and $v_2|W_1$ implies $(v_1,v_2)= (1)$ since $W_1$
and $W_2$ are coprime. This shows that
$\HEF_0(R,W)=\HEF_{0,W_1}(R,W)\vee \HEF_{0,W_2}(R,W)$. By Proposition
\ref{prop:fact}, we have $R$-linear equivalences
$\HEF_{0,W_i}(R,W)\simeq \HEF_0(R,W_i)$ which are bijective on
objects. This implies the conclusion. \qed
\end{proof}

\begin{Definition}
A {\em reduced multiplicative partition} of $W$ is a factorization: 
\be
W=W_1 W_2\ldots W_n
\ee
where $W_1,\ldots, W_n$ are mutually coprime elements of $R$. 
\end{Definition}

\

\begin{Corollary}
\label{cor:coproduct}
Let $W=W_1\ldots W_n$ be a reduced multiplicative partition of
$W$. Then there exists a natural equivalence of $R$-linear categories:
\be
\HEF_0(R,W)\simeq \vee_{i=1}^n \HEF_0(R,W_i)~~.
\ee
\end{Corollary}

\begin{proof}
Follows immediately from Proposition \ref{prop:coproduct}. \qed
\end{proof}

\

\noindent Let $e_v$ be a primary matrix factorization of $W$. Then $v=p^i$ for
some prime divisor $p$ of $W$ and some integer $i\in \{0,\ldots, n\}$,
where $n$ is the order of $p$ as a divisor of $W$.  We have $W=p^n
W_1$ for some element $W_1\in R$ such that $p$ does not divide $W_1$
and $u=p^{n-i}W_1$. Thus $(u,v)=(p^{\min(i,n-i)})$.

\

\begin{Definition}
The prime divisor $p$ of $W$ is called the {\em prime locus} of
$e_v$. The order $n$ of $p$ is called the {\em order} of $e_v$ while
the integer $i\in \{0,\ldots, n\}$ is called the {\em size} of $e_v$.
\end{Definition}

\

\noindent Let $R$ be a B\'ezout domain and $p\in R$ be a prime
element. Fix an integer $n\geq 2$ and consider the quotient ring:
\be
A_n(p)\eqdef R/\langle p^n\rangle~~.
\ee
Let $\bm_n(p)=p A_n(p)=\langle p\rangle /\langle p^n\rangle$ and $\bk_p=R/\langle p\rangle$.

\

\begin{Lemma}
\label{lemma:prime}
The following statements hold:
\begin{enumerate}
\itemsep 0.0em
\item The principal ideal $\langle p\rangle $ generated by $p$ is maximal.
\item The primary ideal $\langle p^n\rangle $ is contained in a unique maximal ideal
  of $R$.
\item The quotient $A_n(p)$ is a quasi-local ring with maximal ideal
  $\bm_n(p)$ and residue field $\bk_p$.
\item $A_n(p)$ is a generalized valuation ring.
\end{enumerate}
\end{Lemma}

\begin{proof}

\

\begin{enumerate}
\itemsep 0.0em
\item Let $I$ be any ideal containing $\langle p\rangle $. If $\langle
p \rangle \neq I$, then take any element $x \in I \setminus \langle p
\rangle$. Then we have the proper inclusion $\langle p \rangle
\subsetneq \langle p,x\rangle$.  Since $R$ is a B\'ezout domain, the
ideal $\langle p,x\rangle$ is generated by a single element $y$. We
have $y|p$, so $y$ is a unit of $R$ since $p$ is prime. Since $y$
belongs to $I$, this gives $I=R$. Thus $\langle p\rangle$ is a maximal
ideal.

\item Let $m$ be a maximal ideal of $R$ containing $\langle p^n\rangle$. Then
$p^n\in m$, which implies $p\in m$ since $m$ is prime. Thus $\langle
p\rangle \subset m$, which implies $m=\langle p\rangle$ since $\langle
p\rangle$ is maximal by point $1$. This shows that $R/\langle
p^n\rangle $ has a unique maximal ideal, namely $\langle p\rangle
/\langle p^n\rangle$.

\item Since $R$ is B\'ezout and $\langle p^n\rangle $ is
finitely-generated, the quotient $R/\langle p^n\rangle$ is a B\'ezout
ring (which has divisors of zero when $n\geq 2$). By point 2. above,
$R/\langle p^n\rangle$ is also a quasi-local ring.

\item Follows from \cite[Lemma 1.3 (b)]{FS} since $R$ is a valuation
ring. \qed
\end{enumerate}
\end{proof}

\

\noindent Recall that an object of an additive category is called {\em
  indecomposable} if it is not isomorphic with a direct sum of two
non-zero objects.  

\

\begin{Proposition}
\label{prop:local}
Let $e_v$ be a primary factorization of $W$ with prime locus $p$,
order $n$ and size $i$. Then $e_v$ is an indecomposable object of
$\hmf(R,W)$ whose endomorphism ring $\End_{\hmf(R,W)}(e_v)$ is a
quasi-local ring isomorphic with $A_{\min(i,n-i)}(p)$.
\end{Proposition}

\begin{proof}
We have $\End_{\hmf(R,W)}(e_v)=R/\langle u,v\rangle =R/\langle p^{\min(i,n-i)}\rangle$ by
Corollary \ref{cor:EndHMF}. This ring is quasi-local by Lemma
\ref{lemma:prime}. Since quasi-local rings have no nontrivial
idempotents, it follows that $e_v$ is an indecomposable object of
$\hmf(R,W)$.  \qed
\end{proof}

\

\begin{Lemma}
\label{lemma:mult}
Let $v_1$ and $v_2$ be two divisors of $W$ which are
mutually coprime. Then $\Hom_{\hmf(R,W)}(e_{v_1},
e_{v_2})=0$.
\end{Lemma}

\begin{proof}
Let $u_i:=W/v_i$. Then $(v_1,v_2,u_1,u_2)=(1)$ since $(v_1,v_2)=(1)$. Thus  $\langle v_1,v_2,u_1,u_2\rangle =R$ and 
the statement follows from Proposition \ref{prop:HomHMF}. \qed
\end{proof}

\

\begin{Proposition}
\label{prop:primary_shift}
Let $p$ be a prime divisor of $W$ of order $n$ and $i\in \{1,\ldots, n\}$. Then: 
\be
\Sigma e_{p^i}\simeq_{\hmf(R,W)} e_{p^{n-i}}~~.
\ee
\end{Proposition}

\begin{proof}
Let $W_1\eqdef p^n$, $W_2\eqdef W/p^n$ and $v\eqdef p^i$, $u\eqdef
W/v=p^{n-i}W_2$.  We have $\Sigma e_{p^i}=\Sigma
e_v=e_{-u}\simeq_{\hmf(R,W)} e_{u}$. Since $p^{n-i}$ is a divisor of
$W_1$ and $(W_1,W_2)=1$, Proposition \ref{prop:isomW1} gives
$e_u=e_{p^{n-i} W_2}\simeq_{\hmf(R,W)} e_{p^{n-i}}$. \qed
\end{proof}

\section{The additive category $\mhef(R,W)$ for a B\'ezout domain and critically-finite $W$}

\noindent Let $R$ be a B\'ezout domain and $W$ be a critically-finite element of $R$. 

\

\begin{Proposition}
\label{prop:decomp}
Let $e_v$ be an elementary factorization of $W$ over $R$ such that
$v=\prod_{i=1}^n v_i$, where $v_i\in R$ are mutually coprime divisors
of $W$.  Then there exists a natural isomorphism in $\hmf(R,W)$:
\be
e_v\simeq_{\hmf(R,W)}\bigoplus_{i=1}^n e_{v_i}~~
\ee
In particular, an elementary factorization $e_v$ for which $v$ is
finitely-factorizable divisor of $W$ is isomorphic in $\hmf(R,W)$ with
a direct sum of primary factorizations.
\end{Proposition}

\begin{proof}
Let $d$ be any divisor of $W$.  By Proposition \ref{prop:HomHMF}, we
have isomorphisms of $R$-modules: 
\be
\Hom_{\hmf(R,W)}(e_d,e_{v_i})\simeq_R R/\alpha_W(d,v_i)~~\mathrm{and}~~\Hom_{\hmf(R,W)}(e_d, e_v)\simeq_R R/\alpha_W(d,v)~~.
\ee
Since $v_i$ are mutually coprime, Proposition \ref{prop:mult} gives
$\alpha_W(d,v)=\prod_{i=1}^n \alpha_W(d,v_i)$, where $\alpha_W(d,v_i)$ are
principal ideals generated by mutually coprime elements. The Chinese reminder theorem gives an 
isomorphism of $R$-modules:
\be
R/\alpha_W(d,v) \simeq_R \bigoplus_{i=1}^n \left[R/\alpha_W(d,v_i)\right] 
\ee
Combining the above, we conclude that there exist natural isomorphisms 
of $R$-modules:  
\ben
\label{varphi_d}
\varphi_d:\Hom_{\hmf(R,W)}(e_d,e_v)\stackrel{\sim}{\longrightarrow} \Hom_{\hmf(R,W)}(e_d, \bigoplus_{i=1}^n e_{v_i})~~,
\een
where we used the fact that $\Hom_{\hmf(R,W)}(e_d,\bigoplus_{i=1}^n
e_{v_i})\simeq_R \bigoplus_{i=1}^n\Hom_{\hmf(R,W)}(e_d,
e_{v_i})$. This implies that the functors $\Hom_{\mhef(R,W)}(-,e_v)$
and $\Hom_{\mhef(R,W)}(-,\oplus_{i=1}^n e_{v_i})$ are isomorphic. By
the Yoneda lemma, we conclude that there exists a natural ismorphism
$e_v\simeq_{\mhef(R,W)} \bigoplus_{i=1}^n e_{v_i}$~. \qed
\end{proof}

\

\noindent Recall that a {\em Krull-Schmidt category} is an additive
category for which every object decomposes into a finite direct
sum of objects having quasi-local endomorphism rings.

\

\begin{Theorem}
\label{thm:Krull} The additive category
$\mhef(R,W)$ is Krull-Schmidt and its non-zero indecomposable objects
are the non-trivial primary matrix factorizations of $W$. In
particular, $\mhef(R,W)$ is additively generated by $\hef_0(R,W)$.
\end{Theorem}

\

\begin{proof}
Suppose that $W$ has the decomposition \eqref{Wcritform}.
Any elementary factorization $e_v$ of $W$ corresponds to a divisor $v$
of $W$, which must have the form $v=v_0 p_{s_1}^{l_1}\ldots
p_{s_m}^{l_m}$, where $1\leq s_1<\ldots < s_m\leq N$ and $1\leq
l_i\leq n_{s_i}$, while $v_0$ is a divisor of $W_0$.  Applying
Proposition \ref{prop:decomp} with $v_i=p_{s_i}^{l_i}$ for $i\in \{1,\ldots, m\}$,
we find $e_v\simeq_{\hmf(R,W)} \oplus_{i=0}^m e_{v_i}$, where we defined
$u_0=W/v_0=\frac{W_0}{v_0} \, p_1^{n_1}\ldots p_N^{n_N}$ and $u_i=W/v_i=W_0
p_1^{n_1}\ldots p_{s_i}^{n_{s_i}-l_i} \ldots p_N^{n_N}$ for
$i\in \{1,\ldots, m\}$. We have $(u_0,v_0)=(u_0,W_0/u_0)$. Since $W_0=u_0v_0$,
it follows that $(u_0,v_0)^2|(W_0)$.  Since $W_0$ has no critical
divisors, we must have $(u_0,v_0)=(1)$ and hence
$e_{v_0}\simeq_{\hmf(R,W)} 0$. For $i\in \{1,\ldots, m\}$, we have
$(u_i,v_i)=p_{s_i}^{\mu_i}$, where $\mu_i\eqdef
\min(l_i,n_{s_i}-l_i)$. Thus $e_{v_i}$ is primary of order $\mu_{s_i}$
when $\mu_{s_i}\geq 1$ and trivial when $\mu_i=0$.  This gives a
direct sum decomposition:
\be
e_v\simeq_{\hmf(R,W)} e_{v_0}\oplus_{i\in \{1,\ldots, m| l_i\leq n_{s_i}-1\}}  e_{p_{s_i}^{l_i}}\simeq \oplus_{i\in \{1,\ldots, m| l_i<n_{s_i}\}}  e_{p_{s_i}^{l_i}}~~,
\ee
where all matrix factorizations in the direct sum are primary except
for $e_{v_0}$. If $l_i=n_{s_i}$ for all $i\in \{1,\ldots, m\}$, then
the sum in the right hand side is the zero object of $\hmf(R,W)$. We
conclude that any elementary matrix factorization decomposes into a
finite direct sum of primary matrix factorizations.  On the other
hand, any matrix factorization of $W$ decomposes as a finite direct
sum of elementary factorizations and hence also as a finite direct sum
of primary factorizations whose prime supports are the prime divisors
of $W$. By Proposition \ref{prop:local}, every primary matrix
factorization has a quasi-local endomorphism ring.  \qed
\end{proof}

\

\begin{Theorem}
\label{thm:decomp}
Suppose that $R$ is a B\'ezout domain and $W$ has the decomposition \eqref{Wcritform}. 
Then there exists an equivalence of categories:
\be
\mhef(R,W) \simeq \vee_{i=1}^N \mhef(R,p_i^{n_i})~~,
\ee
where $\vee$ denotes the coproduct of additive categories. 
\end{Theorem}

\

\begin{proof}
Theorem \ref{thm:Krull} and Proposition \ref{prop:decomp} imply that
$\mathbf{hef}(R,W)$ is additively generated by the additive
subcategories $\mathbf{hef}_{p_i^{n_i}}(R,W)\simeq \mhef(R,p_i^{n_i})$,
where we used Proposition \ref{prop:fact}.  These categories are
mutually orthogonal by Lemma \ref{lemma:mult}.
\qed
\end{proof}

\subsection{A conjecture}

\noindent Consider the inclusion functor:

\be
\iota: \mhef(R,W) \to \hmf(R,W)
\ee

\

\begin{Conjecture}
\label{conj:main}
The inclusion functor $\iota$ is an equivalence of $R$-linear categories.
\end{Conjecture}

\

\noindent Conjecture \ref{conj:main} and Theorem \ref{thm:Krull} imply: 

\

\begin{Conjecture}
\label{conj:KS}
Let $R$ be a B\'ezout domain and $W$ be a critically-finite element of
$R$. Then $\hmf(R,W)$ is a Krull-Schmidt category.
\end{Conjecture}

\

\noindent In \cite{edd}, we establish Conjecture \ref{conj:main} for
the case when $R$ is an elementary divisor domain. This shows that
Conjecture \ref{conj:main} is implied by the still unsolved conjecture
\cite{Helmer2} that any B\'ezout domain is an elementary divisor domain.
Some recent work on that conjecture can be found in \cite{Lorenzini}.

\section{Counting elementary factorizations}
\label{subsec:counting}

\noindent In this section, we give formulas for the number of
isomorphism classes of objects in the categories $\HEF(R,W)$ and
$\hef(R,W)$ when $W$ is critically-finite.

\subsection{Counting isomorphism classes in $\HEF(R,W)$}
Let $W=W_0 W_c$ be a critically-finite element of $R$, where $W_0\in R$ is
non-critical and $W_c=p_1^{n_1}\ldots p_r^{n_r}$ with
prime $p_j\in R$ and $n_j\geq 2$ (see Definition
\ref{def:crit-fin}). Let $\cHef(R,W)$ denote the set of isomorphism
classes of objects in the category $\HEF(R,W)$.  We are interested in
the cardinality:
\be
N(R,W)\eqdef |\cHef(R,W)|
\ee
of this set. In this subsection, we derive a formula for $N(R,W)$ as a
function of the orders $n_i$ of the prime elements $p_i$ arising in
the prime decomposition of $W_c$. The main result of this subsection
is Theorem \ref{thm:N} below.

\

\begin{Lemma}
\label{lemma:NeWc}
The cardinality $N(R,W)$ depends only on the critical part $W_c$ of $W$.
\end{Lemma}

\begin{proof}
Let $W=puv$ with a divisor $p$ coprime with both
$u$ and $v$. Taking $b=p$ and $c=1$ in Proposition \ref{prop:EFcrit1}
gives:
\begin{equation}
\left[\begin{array}{cc} 0 & pv\\ u &
    0\end{array}\right]\simeq \left[\begin{array}{cc} 0 & v\\ pu &
    0\end{array}\right] ~~.
\end{equation}
Together with Corollary \ref{prop:similar}, this implies $N(R, W)=N(R, W_c)$. 
\qed
\end{proof}

\

\noindent From now on, we will assume that $W\in R$ is fixed and is of the form:
\ben\
\label{f129} 
W=W_c=p_1^{n_1} p_2^{n_2} \ldots p_r^{n_r} ~~.
\een
To simplify notations, we will omit to indicate the dependence of some quantities on $W$.

\

\begin{Definition}
Let $T$ be a non-empty set. A map $f:\Ob \EF(R,W)\rightarrow
T$ is called an {\em elementary invariant} if $f(e_1)=f(e_2)$ for any
  $e_1,e_2\in \Ob\EF(R,W)$ such that $e_1\simeq_{\HEF(R,W)} e_2$.
An elementary invariant $f$ is called {\em complete} if the map
$\underline{f}:\cHef(R,W)\rightarrow T$ induced by $f$ is injective.
\end{Definition}

\

\noindent To determine $N(R,W)$, we will construct a complete
elementary invariant. Let:
\be
I\eqdef \{1,\dots,r\}~~,
\ee
where $r$ is the number of non-associated prime factors of $W$, up to association in divisibility. 

\paragraph{Similarity classes of elementary factorizations and normalized divisors of $W$}
Let $\HEF_\rsim(R,W)$ be the groupoid having the same objects as
$\HEF(R,W)$ and morphisms given by similarity transformations of
elementary factorizations and let $\cHef_\rsim(R,W)$ be its set of
isomorphism classes. Since the similarity class of an elementary
factorization $e_v$ is uniquely determined by the principal ideal
$\langle v \rangle$ generated by the divisor $v$ of $W$, the map
$e_v\rightarrow \langle v\rangle$ induces a bijection:
\be
\cHef_\rsim(R,W)\simeq \Div(W)~~,
\ee
where:
\be
\Div(W)\eqdef\big\{\langle v\rangle,\big|\, v|W\big\}=\big\{ \langle v\rangle\,\big|\, v\in R: W\in \langle v\rangle\big\}
\ee
is the set of principal ideals of $R$ containing $W$. Let: 
\ben
\Div_1(W)\eqdef \{\,\prod_{i\in I} p_i^{k_i} \, \big|\, \forall i\in I: k_i\in \{0,\ldots, n_i\}\,\}~~,
\een
be the set of {\em normalized divisors} of $W$. The map $v\rightarrow
\langle v\rangle$ induces a bijection between $\Div_1(W)$ and
$\Div(W)$. Indeed, any principal ideal of $R$ which contains $W$ has a
unique generator which belongs to $\Div_1(W)$, called its {\em
  normalized} generator.  Given any divisor $v$ of $W$, its {\em
  normalization} $v_0$ is the unique normalized divisor $v_0\in
\Div_1(W)$ such that $\langle v\rangle=\langle v_0\rangle$.  Given two
divisors $t,s$ of $W$, their {\em normalized greatest common divisor}
is the unique normalized divisor $(t,s)_1$ of $W$ which generates the
ideal $Rt+Rs$. The {\em set of exponent vectors} of $W$ is defined
through:
\be
A_W\eqdef \prod_{i=1}^r \{0,\ldots, n_i\}~~.
\ee
The map $A_W\ni \bk=(k_1,\ldots, k_r) \mapsto \prod_{i\in I}
p_i^{k_i}\in \Div_1(W)$ is bijective, with inverse
$\bmu:\Div_1(W)\rightarrow A_W$ given by:
\be 
\bmu(v)=(\ord_{p_1}(v),\ldots, \ord_{p_n}(v))=(k_1,\ldots,
k_n)~~\mathrm{for}~~v=\prod_{i\in I} p_i^{k_i}\in \Div_1(W)~~.  
\ee
Combining everything, we have natural
bijections:
\be
\cHef_\rsim(R,W)\simeq \Div(R,W)\simeq \Div_1(R,W)\simeq A_W~~.
\ee

\begin{remark}\label{rem:div_1} In Proposition \ref{prop:EFcrit2}, the
quantity $s_1$ was an arbitrary element of the class $(u_1,v_1)$ for
$e_{v_1}$. For a fixed critically-finite $W$, we have a canonical
choice for this quantity, namely the normalized gcd of $u_1$ and
$v_1$. Thus we define $s(e_1)=(u_1,v_1)_1$. The two definitions are
connected by the relation $(s_1)=(s(e_1))$. Below we introduce
``normalized'' quantities $x(e), y(e)$ which belong to the same
classes in $R^\times/U(R)$ as the quantities $x$ and $y$ defined in
Section 2. The results of Section 2 hold automatically for these
normalized choices.
\end{remark}

\

\noindent Given $t\in \Div(W)$, its {\em index set} is the subset of $I$
given by:
\be
I(t)\eqdef \supp \bmu(t_0)=\big\{i\,\in I  ~\big|~ p_i|t\,\big\}~~.
\ee
Notice that $I(t)$ depends only on the principal ideal $\langle
t\rangle$, that in turn depends only on the class $(t)\in R^{\times} /
U(R)$. This gives a map from $\Div(W)$ to the power set $\cP(I)$ of
$I$. Note that $(t)=(1)$ iff $I(t)=\emptyset$.

\paragraph{The essence and divisorial invariant of an elementary factorization}
Consider an elementary factorization $e$ of $W$ and let $v(e)\in
\Div_1(W)$ be the unique normalized divisor of $W$ for which $e$ is
similar to $e_v$. Let $u=u(e)\eqdef W/v\in \Div_1(W)$ and let
$s(e)\eqdef (v,u)_1\in \Div_1(W)$ be the normalized greatest common
denominator of $v$ and $u$. Let $I_s(e)\eqdef I(s(e))$. Let
$m_i(e)\eqdef \ord_{p_i} (s(e))$ and
$\bm(e)=\bmu(s(e))=(m_1(e),\ldots, m_n(e))$. Then $I_s(e)=\supp
\bm(e)$ and:
\ben
\label{sW}
s(e)=\prod_{i\in I_s(e)} p_i^{m_i(e)}~~.
\een
Let $v'(e)\eqdef v/s(e)$ and $u'(e)\eqdef u/s(e)$. Then
$(v'(e),u'(e))=(1)$ and $W=v(e)u(e)=v'(e)u'(e)s(e)^2$. Define:
\begin{equation}
x(e)\eqdef (s(e),v'(e))_1~~,~~ y(e)\eqdef (s(e),u'(e))_1~~
\end{equation}
and: 
\be 
I_x(e)\eqdef I(x(e))~,~I_y(e)\eqdef I(y(e)) ~~.
\ee
Notice that $(x(e),y(e))_1=1$, thus $I_x(e)\cap
I_y(e)=\emptyset$. Defining $v''(e)\eqdef v'(e)/x(e)$ and
$u''(e)\eqdef u'(e)/y(e)$, we have:
\be
W=  x(e)y(e) v''(e)u''(e)s(e)^2~~,
\ee
where $v''(e), u''(e)$ and $s(e)$ are mutually coprime. Moreover, we have:
\be
\ord_{p_i}(v'(e))=n_i-2m_i(e)~~\forall i\in I_{x}(e)~~\mathrm{and}~~\ord_{p_i}(u'(e))=n_i-2m_i(e)~~\forall i\in I_{y}(e)~~,
\ee
which implies:
\beqan
\label{xy}
\ord_{p_i}x(e)&=&\max (m_i(e),n_i-2m_i(e)) ~~\text{for}~~i\in I_x(e) ~,\nn\\
\ord_{p_i}y(e)&=&\max (m_i(e),n_i-2m_i(e))~~\text{for}~~i\in I_y(e) ~ .
\eeqan
Notice that $\ord_{p_i} x(e)y(e) \equiv n_i \,\mod\, 2$  for $i\in I_x\cup I_y$ if $3m_i<n_i$.

\

\begin{Definition}
The {\em essence} $z:=z(e)$ of an elementary factorization $e$ of $W$
is the normalized divisor of $W$ defined through:
\ben
\label{chs}
z(e)\eqdef \prod_{I_z(e)} p_i^{m_i(e)} ~,~~\text{where}~~~ I_z(e)\eqdef I(s)\backslash \big(I_x(e)\cup I_y(e)\big)~~.
\een
An elementary factorization $e$ is called {\em essential} if
$z(e)=1$, i.e. if $I_z(e)=\emptyset$. 
\end{Definition}

\

\noindent The divisor $s$ defines $x$, $y$ and $z$ uniquely by
\eqref{xy} and \eqref{chs}. These 3 divisors in turn also define $s$
uniquely that can be seen by the inverting the max functions above:
\ben
\ord_{p_i}s = m_i=\begin{cases}
\ord_{p_i} x(e)~~~\quad\quad\text{if}~~ i\in I_x(e)~~\text{and}~~ 3 \,\ord_{p_i}x(e)\geq n_i\\
(n_i - \ord_{p_i}x(e))/2~~~\text{if}~~ i\in I_x(e)~~\text{and}~~ 3 \,\ord_{p_i}x(e)<n_i\\
\ord_{p_i} y(e)~~~\quad\quad \text{if}~~ i\in I_y(e)~~\text{and}~~ 3 \,\ord_{p_i}y(e)\geq n_i\\
(n_i - \ord_{p_i}y(e))/2~~~\text{if}~~ i\in I_y(e)~~\text{and}~~ 3\,\ord_{p_i}y(e)<n_i\\
\ord_{p_i} z(e),~~~\quad\quad \text{if}~~ i\in I_z(e)~~.
\end{cases}
\een
\
\

\noindent The fundamental property of an essential factorization of $e$ is
the equality of sets $I_s(e)=I_x(e)\sqcup I_y(e)$, which will allow us
to compute the number $N_\emptyset(R,W)$ of isomorphism classes of
such factorizations (see Proposition \ref{prop:Ss_car} below). Then $N(R,W)$
will be determined by relating it to $N_\emptyset$ for various
reductions of the potential $W$.

\

Notice that the essence $z(e)$ is a
critical divisor of $W$ and that we have
$(z(e),v'(e))_1=(z(e),u'(e))_1=1$. Since $W=v(e)u(e)=v'(e)u'(e)s(e)^2$,
this gives:
\ben
\label{f130}
\ord_{p_i} W=n_i = 2 m_i(e)=2\ord_{p_i} z(e) ~~\text{for any}~~ i\in I_z(e)~~.
\een

\

\begin{Definition}
The {\em divisorial invariant} of an elementary factorization $e$ of
$W$ is the element $h(e)$ of the set $\Div_1(W)\times
\Sym^2(\cP(I))$ defined through:
\be
h(e)=(s(e),\{I_x(e),I_y(e)\})~~.
\ee
This gives a map $h:\EF(R,W)\rightarrow \Div_1(W)\times \Sym^2(\cP(I))$. 
\end{Definition} 

\

\noindent We have already given a criterion for two elementary
factorizations of $W$ to be isomorphic in Proposition
\ref{prop:EFcrit2}.  There exists another way to characterize when two
objects of $\HEF(R,W)$ (and also of $\hef(R,W)$) are isomorphic, which
will be convenient for our purpose.

\

\begin{Proposition}
\label{prop:ssII} 
Consider two elementary factorizations of $W$. The following
statements are equivalent:
\begin{enumerate}[1.]
\itemsep 0.0em
\item The two factorizations are isomorphic in $\HEF(R,W)$
  {\rm (}respectively in $\hef(R,W)$\rm{)}.
\item The two factorizations have the same $(s,\{x,y\})$ {\rm (}respectively
  same $(s,x,y)${\rm )}.
\item The two factorizations have the same divisorial invariant
  $(s,\{I_x,I_y\})$ {\rm (}respectively same $(s,I_x,I_y)${\rm )}.
\end{enumerate}
In particular, the divisorial invariant $h:\Ob \EF(R,W)\rightarrow
\Div_1(W)\times \Sym^2(\cP(I))$ is a {\em complete} elementary
invariant.
\end{Proposition}

\begin{proof} 
The equivalence between $1.$ and $2.$ follows from Proposition
\ref{prop:EFcrit2}. Indeed, the proposition shows that for two
isomorphic factorizations $e_1$ and $e_2$ the corresponding $s_1$ and
$s_2$ are similar: $(s_1)= (s_2)$ in the notations of Section 2. We
compute $(s_2)= (u_2,v_2)= ((u_2)_1,(v_2)_1)=
((u(e_2),v(e_2))_1)=(s(e_2))$ with the last $s(e_2)$ defined in $\Div_1$ by
\eqref{sW}. Similarly $(s_1)= (s(e_1))$. By the very definition of
$\Div_1$ we have $(s(e_1))=(s(e_2))$ implies $s(e_1)=s(e_2)$.

The implication $2.\Rightarrow 3.$ is obvious.  Thus it suffices to
prove that $3.$ implies $2.$ For this, let $e_1$ and $e_2$ be the two
elementary factorizations of $W$.  Assume that $s(e_1)=s(e_2)$ and
$\{I_x(e_1),I_y(e_1)\}=\{I_x(e_2),I_y(e_2)\}$ and let
$s:=s(e_1)=s(e_2)=\prod_{i\in I(s)} p_i^{m_i}$. Consider the case
$I_x(e_1)=I_x(e_2)$ and $I_y(e_1)=I_y(e_2)$. Applying \eqref{xy} to
$v=v(e_1)$ and $v=v(e_2)$ and using the relations $m_i(e_1)=\ord_{p_i}
s =m_i(e_2)$ gives:
\be
x(e_1)=x(e_2) ~~\text{and}~~y(e_1)=y(e_2)~~.
\ee
When $I_x(e_1)=I_y(e_2)$ and $I_x(e_2)=I_y(e_1)$, a similar argument
gives $x(e_1)=y(e_2)$ and $y(e_1)=x(e_2)$. \qed
\end{proof}

\

\begin{Proposition}\label{prop:zeleminvar}
The map $z:\Ob \EF(R,W)\rightarrow \Div_1(W)$ which gives the essence of
an elementary factorization is an elementary invariant.
\end{Proposition}
\begin{proof}
Let $e_1$ and $e_2$ be two factorizations of $W$ which are isomorphic
in $\HEF(R,W)$. By Proposition \ref{prop:ssII}, we have $s(e_1)=s(e_2)$ and
$\{I_x(e_1),I_y(e_1)\}=\{I_x(e_2),I_y(e_2)\}$. Hence:
\ben
I(z(e_1))=I(s(e_1))\backslash \big(I_x(e_1)\cup I_y(e_1)\big)=I(s(e_2))\backslash \big(I_x(e_2)\cup I_y(e_2)\big)=I(z(e_2))~~.
\een
Applying \eqref{f130} for $e=e_1$ and $e=e_2$ gives $\ord_{p_i}
z(e_1)=\ord_{p_i} z(e_2)$ for any $i\in I(z(e_1))=I(z(e_2))$. Thus
$z(e_1)=z(e_2)$. \qed
\end{proof}

\paragraph{The essential reduction of an elementary factorization}
For any normalized critical divisor $z$ of $W$, let $\HEF_z(R,W)$
denote the full subcategory of $\HEF(R,W)$ consisting of those
elementary factorizations whose essence equals $z$ and let
$\cHef_z(R,W)$ be its set of isomorphism classes. Then $\cHef_1(R,W)$
consists of the isomorphism classes of essential factorizations.

\

\begin{Definition}\label{def:essred}
The {\em essential reduction} of an elementary factorization $e:=e_v$ of $W$
is the essential elementary factorization of $W/z(e)^2$ defined through: 
\be
\essred(e)\eqdef e_{v/z(e)}~~.
\ee
This gives a map $\essred:\Ob \EF(R,W)\rightarrow \Ob \HEF_1(R,W/z(e)^2)$. 
\end{Definition}

\

\noindent To see that $\essred$ is well-defined, consider the
elementary factorization $\tilde{e}=e_{v/z(e)}$:
\ben
\label{Wred}
\widetilde{W}\eqdef W/z(e)^2=u(\tilde{e})v(\tilde{e})~~,
\een
where $v(\tilde{e})=v(e)/z(e)$ and
$u(\tilde{e})=u(e)/z(e)$. We compute:
\be 
s(\tilde{e})\eqdef
(v(\tilde{e}), u(\tilde{e}))_1=(v(e)/z(e), u(e)/z(e))_1=s(e)/z(e)
\ee
 and $v'(\tilde{e})\eqdef v(\tilde{e})/s(\tilde{e})=v'(e)$,
 $u'(\tilde{e})\eqdef u(\tilde{e})/s(\tilde{e})=u'(e)$. Thus
 $x(\tilde{e})\eqdef (s(\tilde{e}),v'(\tilde{e}))_1=x(e)$ and
 $y(\tilde{e})\eqdef (s(\tilde{e}),u'(\tilde{e}))_1=y(e)$. By
 \eqref{chs} applied to $\tilde{e}$ and $e$, we derive
 $I_z(\tilde{e})=I_s(\tilde{e})\backslash\big( I_x(\tilde{e})\cup
 I_y(\tilde{e})\big)=I_z(e)\cup I_s(e) \backslash\big( I_x(e)\cup
 I_y(e)\big)=\emptyset$, which implies $z(\tilde{e})=1$.  Hence
 $\essred(e)$ is an essential elementary matrix factorization of
 $\widetilde{W}$. Also notice the relation:
\be
\big(z(e), W/z(e)^2\big)=(1)~~,
\ee
which follows from the fact that $z(e)$ is coprime with $v'(e)$ and
$u'(e)$.

\

\begin{Lemma}
\label{lemma:chs2Iso}
For any critical divisor $z$ of $W$ such that
  $(z,W/z^2)\sim 1$, the map $\essred$ induces a well-defined bijection
$essred_z:\cHef_z(R,W)\stackrel{\sim}{\rightarrow} \cHef_1(R,W/z^2)$\,.
\end{Lemma}

\begin{proof}
We perform the proof in two steps:
\begin{enumerate}
\item  Let $e_1$ $:=e_{v_1}$ and $e_2$:=$e_{v_2}$ be two elementary
  factorizations of $W$ such that $z(e_1)=z(e_2)=z$ and let
  $v_1=v(e_1)$, $v_2=v(e_2)$. Define also $v_3\eqdef v_1/z$ and
  $v_4\eqdef v_2/z$. To show that $\essred_z$ is well-defined, we have
  to show that $e_1\simeq_{\HEF(R,W)} e_2$ implies that the two
  essential elementary factorizations $e_3:= e_{v_3}$ and $e_4:=
  e_{v_4}$ of $\widetilde{W}\eqdef W/z^2$ are isomorphic in
  $\HEF(R,\widetilde{W})$. For this, we compute:
\be
s(e_{1})\eqdef (v_1,u_1)_1=(z(e_{1})\cdot v_3,z(e_{1})\cdot u_3)_1=z(e_{1})\cdot (v_3,u_3)_1=z(e_{1})\cdot s(e_3)~~.
\ee
Thus:
\be
x(e_{1})\eqdef (s(e_{1}),v'_1)_1=(z(e_{1})\cdot s(e_3),z(e_{1}) v_3/z(e_{1})s(e_3))_1=(s(e_3),v'_3)_1=x(e_3).
\ee 
The third equality above holds since $(z(e_{1}),v'_1)_1=1$ and thus
$(z(e_{1}),v_3)_1=1$. Similarly, we have $s(e_{2})=z(e_2) \cdot
s(e_{4})$ and we find $y(e_{1})=y(e_3)$ as well as $x(e_{2})=x(e_4)$
and $y(e_{2})=y(e_4)$. By Proposition \ref{prop:ssII}, the condition
$e_1\simeq_{\HEF(R,W)} e_2$ implies $s(e_1)=s(e_2)=s$ and
$I(s(e_1))=I(s(e_2))$, thus $z(e_1)=z(e_2)=z$. If
$\big(s(e_1),\{x(e_1),y(e_1)\}\big) =
\big(s(e_2),\{x(e_2),y(e_2)\}\big)$, then
$\big(s(e_3),\{x(e_3),y(e_3)\}\big) =
\big(s(e_4),\{x(e_4),y(e_4)\}\big)$. Thus $e_3\simeq e_4$.
\item Let $z$ be a critical divisor of $W$ such that
  $(z,W/z^2)=(1)$. For any essential elementary factorization $e_v$ of
  $\widetilde{W}$, the elementary factorization $e_{zv}$ of $W$ is an
  object of $\HEF_z(R,W)$ and we have $\essred(e_{vz})=e_v$. This
  shows that $\essred_z$ is surjective.  Now let $e_3$ and $e_4$ be
  two essential elementary factorizations of $\widetilde{W}\eqdef
  W/z^2$ which are isomorphic in $\HEF(R,\widetilde{W})$. Let
  $v_1\eqdef z v_3$ and $v_2\eqdef z v_4$. To show that $\essred_z$ is
  injective, we have to show that the two elementary factorizations
  $e_1:=e_{v_1}$ and $e_2:=e_{v_2}$ of $W$ are isomorphic in
  $\HEF(R,W)$. For this, notice that $\big(s(e_3),
  \{x(e_3),y(e_3)\}\big) = \big(s(e_4),\{x(e_4),y(e_4)\}\big)$ by
  Proposition \ref{prop:ssII}. This implies $\big(s(e_1),
  \{x(e_1),y(e_1)\}\big) = \big(s(e_2),\{x(e_2),y(e_2)\}\big)$, with
  $z(e_1)=z(e_2)=z$. Hence $e_1$ and $e_2$ are isomorphic in
  $\HEF(R,W)$ by the same proposition. \qed
\end{enumerate}
\end{proof}

\paragraph{A formula for $N(R,W)$ in terms of essential reductions}
Let:
\ben
S \eqdef \im\, h\subset \Div_1(W)\times \Sym^2(\cP(I))~~.
\een
The degrees of the prime factors $p_i$ in the decomposition
 \eqref{f129} of $W$ define on $I=\{1,\dots,r\}$ a $\Z_2$-grading
 given by:
\ben\label{f150}
I^\0\eqdef \{i\in I \,\big| \, n_i~\mathrm{is~even}\}~,~I^\1\eqdef \{i\in I\,\big|\, n_i~\mathrm{is~odd}\}~.
\een
Let:
\be
r^{\0}\eqdef |I^{\0}|~~,~~r^{\1}\eqdef |I^{\1}|~~. 
\ee
Since $I=I^\0\sqcup I^\1$, we have $r=r^\0+r^\1$. Any non-empty subset
$K\subset I$ is endowed with the $\Z_2$-grading induced from $I$.
For any critical divisor $z$ of $W$, we have $z^2|W$, which implies
$I(z)\subset I^\0$. For any subset $J\subset I^\0$, define:
\ben
\label{f151}
z_J\eqdef \prod_{i\in J} p_i^{n_i/2}~~,
\een
which is a normalized critical divisor of $W$ satisfying $(z_J, W/z_J^2)_1=1$. 
Also define:
\be
S_J\eqdef h(\cHef_{z_J}(R,W))\subset S~~.
\ee
and:
\ben
\label{f121}
N_J(R,W)\eqdef |S_{J}|~~.
\een
Since $h$ is a complete elementary invariant, we have
$N_\emptyset(R,W)=|h(\cHef_1(R,W))|=|\cHef_1(R,W)|$. Moreover, Lemma
\ref{lemma:chs2Iso} gives:
\ben
\label{NJN0}
N_J(R,W)=N_\emptyset(R,W/z_J^2)~~.
\een

\

\begin{Proposition}
\label{prop:N1toN}
We have: 
\ben
\label{f83}
N(R,W) = \sum_{J\subset I^{\0}} N_\emptyset (R,W/z_J^2)~~.
\een
\end{Proposition}

\begin{proof}
Follows immediately from Lemma \ref{lemma:chs2Iso} and the remarks above. \qed.
\end{proof}

\paragraph{Computation of $N_\emptyset(R,W)$} 
Notice that $z_\emptyset=1$. Since $h$ is a complete elementary
invariant, we have $N_\emptyset(R,W)=|S_\emptyset|$, where:
\be
S_\emptyset=\{h(e) \,\big|\, e\in \Ob \HEF(R,W): z(e)=1\}~~.
\ee
We will first determine the cardinality of the set: 
\be
S_{\emptyset,k}\eqdef \{h(e) \,\big|\, e\in \Ob \HEF(R,W): z(e)=1~\mathrm{and}~ |I_s(e)|=k\}~~.
\ee
We have: 
\be
S_\emptyset=\sqcup_{k=1}^r S_{\emptyset,k}~~.
\ee

\begin{Lemma}
\label{lemma:N0}
For $k\geq 1$, we have:
\ben\label{f120}
|S_{\emptyset,k}|=2^{k-1}\cdot \sum_{\substack{K\subset I,\\ |K|=k}}\prod_{i\in K} \big\lfloor\frac{n_i-1}{2}\big\rfloor ~~.
\een
\end{Lemma}
\begin{proof}
Consider a subset $K\subset I(s)$ of cardinality $|K|=k$. 
Since $s^2|W$, we have:
\be
1\leq m_i(e_v)\leq \lfloor(n_i-1)/2\rfloor~~\forall i\in I(s)~~.
\ee
There are $\prod_{i\in K}\big\lfloor\frac{n_i-1}{2}\big\rfloor$
different possibilities for $s$ such that $I(s)=K$. We can have
several elements $(s,\{x,y\})$ of $S_{1,k}$ with the same $s$ since
$x$ and $y$ can vary. This is where the coefficient $2^{k-1}$ in front
in \eqref{f120} comes from, as we now explain. Fixing the set $I(s)$ with $|I(s)|=k$, 
we have a set $\cP(I(s))$ of $2^k$ 
partitions $I(s)=I_x\sqcup I_y$ as disjoint union of 2 sets. These can
be parameterized by the single subset $I_v\subset I(s)$ since $I_u=
I(s)\backslash I_v$. Define: 
\be
S_{\emptyset,k,s}=\{h(e) \,\big|\, e\in \Ob \HEF(R,W): z(e)=1~\mathrm{and}~ |I_s(e)|=k~\mathrm{and}~~s(e)=s\}~~.
\ee
Consider the surjective map
\be 
\Psi : \cP(I(s)) \rightarrow S_{\emptyset,k,s} 
\ee
which sends a partition $\beta=(I_1,I_2)$ of $I(s)$ to the element
$\alpha=(s,\{I_1,I_2\})$.  The preimage $\Psi^{-1}(h(e))$ of an
element $h(e)\in S_{\emptyset,k,s}$ consist of two elements :
$(I_1,I_2)$ and $(I_2,I_1)$. Thus the map is 2:1. This holds for every
$K$ with $|K|=k$. Comparing the cardinalities of $\cP(I(s))$ and
$S_{\emptyset,k,s}$, we find:
\be
|S_{\emptyset,k,s}|=|\cP(I(s))|/2=2^k/2 =2^{k-1}~~.
\ee
This holds for any $s$ with $I(s)=K$, where $K\subset I$ has
cardinality $k$. Since $S_{\emptyset,k}=\sqcup_s S_{\emptyset ,k, s}$ and since
the cardinality $|S_{\emptyset,k,s}|$ does not depend on $s$, we find:
\be
|S_{\emptyset,k}|=\sum_s |S_{\emptyset ,k,s}|= 2^{k-1} \sum_{K\subset I: |K|=k}\prod_{i\in K} \big\lfloor\frac{n_i-1}{2}\big\rfloor ~~. 
\ee
\qed
\end{proof}

\

\begin{Proposition}
\label{prop:Ss_car}
With the definitions above, we have: 
\ben \label{f126}
N_\emptyset(R,W)=|S_\emptyset|=1+ \sum_{k=1}^r 2^{k-1}\sum_{\substack{K\subset I\\|K|=k }}  \prod_{i\in K} \big\lfloor\frac{n_i-1}{2}\big\rfloor ~~.
\een
\end{Proposition}

\begin{proof}
Since $S_\emptyset=\sqcup_{k=1}^r S_{\emptyset,k}$, the the previous lemma gives:
\ben 
|S_\emptyset|=1 + \sum_{k=1}^r|S_{\emptyset,k}|=1+ \sum_{k=1}^r 2^{k-1}\sum_{\substack{K\subset I\\|K|=k }}  \prod_{i\in K} \big\lfloor\frac{n_i-1}{2}\big\rfloor~~.
\een
The term $1$ in front corresponds to the unique element
$(1,\{\emptyset, \emptyset\})$ of $S$.  \qed
\end{proof}

\paragraph{Computation of $N(R,W)$}

The main result of this subsection is the following:

\

\begin{Theorem}
\label{thm:N}
The number of isomorphism classes of $\HEF(R,W)$ for a critically-finite $W$ as in \eqref{f129} is given by:
\ben
\label{Nfinal}
N(R,W)=2^{r^\0}+ \sum_{k=0}^{r^\1} 2^{r^{\0}+k-1} \sum_{\substack{K\subsetneq I \\|K^\1|=k}}\prod_{i\in K} \big\lfloor\frac{n_i-1}{2}\big\rfloor ~~.
\een
\end{Theorem}
\begin{proof}
Combining Proposition \ref{prop:N1toN} and Proposition \ref{prop:Ss_car},
we have:
\ben
\label{f125}
N(R,W) =   \sum_{J\subset I^{\0}} \big( 1+ \sum_{k=1}^{r-j} 2^{k-1}\sum_{\substack{K\subset I\backslash J\\|K|=k }}  \prod_{i\in K} \big\lfloor\frac{n_i-1}{2}\big\rfloor \big)~~,
\een
where $j\eqdef |J|$. We will simplify this expression by changing the
summation signs and applying the binomial formula.

Since $r_0=|I_0|$ and $J\subset I^\0$, we have $j\leq r_0$. 
For fixed $j$, we have $r_0
\choose j$ different subsets $J\subset I^{\0}$ of this
cardinality. The contribution to $N_\emptyset(R,\widetilde{W})$ of any such $J$
has the free coefficient $1$. Then the free coefficient of
$N(R,W)$ is:
\begin{equation}
\label{f127}
\sum_{j=0}^{r^\0} {r^\0 \choose j}= 2^{r^\0}~~.
\end{equation}
For the other coefficients of \eqref{f125}, we consider a subset
$K\subset I$ as in \eqref{f126}. Such an index set $K=K^{\1}\sqcup
K^{\0}\subset I$ of cardinality $k=k^{\1} + k^{\0}$ appears in
$N_\emptyset(R,\widetilde{W})$ if $K^{\0}\subset I^{\0}\backslash J$. The
coefficient of $\prod_{i\in K} \big\lfloor\frac{n_i-1}{2}\big\rfloor$
in $N(R, \widetilde{W})$ is $2^{k-1}$ for any of $r^{\0}-k^{\0}
\choose j$ choices of $J$. It follows that this coefficient in
$N(R,W)$ is:
\ben\label{f128}
\sum_{j=0}^{r^{\0}-k^{\0}} 2^{k-1}{r^{\0}-k^{\0} \choose j } = 2^{k^\0+k^\1-1} 2^{r^{\0}-k^{\0}}  \\ =
  2^{r^{\0}+k^{\1}-1}~~.
\een
Together with \eqref{f127} and \eqref{f128}, relation \eqref{f125} gives:
\be
N(R,W)=2^{r^\0}+ \sum_{\substack{\emptyset\subset K\subsetneq I \\|K^\1|=k^\1}} 2^{r^{\0}+k^{\1}-1} \prod_{i\in K} \big\lfloor\frac{n_i-1}{2}\big\rfloor ~~,
\ee
which is equivalent to \eqref{Nfinal}. 
\qed
\end{proof}

\

\noindent The two examples below illustrate how the coefficients behave for  $r=2$ and $r=3$.

\begin{example}
Let $W=p_1^n p_2^m$ for prime elements $p_1,p_2\in R$ such that
$(p_1)\neq (p_2)$ and $n,m\geq 2$ with odd $n$ and $m$. Then:
\be
N(R,W)=1+\big\lfloor\frac{n-1}{2}\big\rfloor+\big\lfloor\frac{m-1}{2}\big\rfloor~~.
\ee
\end{example}

\begin{example}
Consider $W=p_1^n p_2^m p_3^k$ for primes $p_1,p_2,p_3\in R$ which are mutually non-associated in divisibility
and orders $n,m,k\geq 2$ subject to the condition that $n$ and $m$ are
even while $k$ is odd. Then:
\begin{multline}
N(R,W)=2+2\big\lfloor\frac{n-1}{2}\big\rfloor+ 2\big\lfloor\frac{m-1}{2}\big\rfloor+\big\lfloor\frac{k-1}{2}\big\rfloor + \\
4\big\lfloor\frac{n-1}{2}\big\rfloor\big\lfloor\frac{m-1}{2}\big\rfloor+ 2\big\lfloor\frac{n-1}{2}\big\rfloor\big\lfloor\frac{k-1}{2}\big\rfloor + 2\big\lfloor\frac{n-1}{2}\big\rfloor\big\lfloor\frac{m-1}{2}\big\rfloor  ~~.
\end{multline}
\end{example}

\subsection{Counting isomorphism classes in $\hef(R,W)$}

We next derive a formula for the number of isomorphism classes in the
category $\hef(R,W)$ for a critically-finite $W$ (see Theorem \ref{thm:Nc}
below). Since the morphisms of $\hef(R,W)$ coincide with the even
morphisms of $\HEF(R,W)$, the number of isomorphism classes of
$\hef(R,W)$ is larger than $N(R,W)$. The simplest difference between
the two cases arises from the fact that suspension does not preserve
the ismorphism class of an elementary factorization in the category
$\hef(R,W)$. Let $\ccHef(R,W)$ be the set of isomorphism classes of
objects in $\hef(R,W)$ and:
\be
\chN(R,W)= |\ccHef(R,W)|~~.
\ee

\

\begin{Lemma}
\label{lemma:NeWcc}
The cardinality $\chN(R,W)$ depends only on the critical part $W_c$ of $W$.
\end{Lemma}
\begin{proof}
The proof is identical to that of Lemma \ref{lemma:NeWc}.
\end{proof}

\

\begin{Definition}
Let $T$ be a non-empty set. A map $f:\Ob \EF(R,W)\rightarrow
T$ is called an {\em even elementary invariant} if $f(e_1)=f(e_2)$ for any
  $e_1,e_2\in \Ob\EF(R,W)$ such that $e_1\simeq_{\hef(R,W)} e_2$.
An even elementary invariant $f$ is called {\em complete} if the map
$\underline{f}:\ccHef(R,W)\rightarrow T$ induced by $f$ is injective.
\end{Definition}

\

\noindent As in the previous subsection, we will compute $\chN(R,W)$
by constructing an {\em even} complete elementary invariant.

\

\begin{Definition}
The {\em even divisorial invariant} of an elementary factorization $e$ of
$W$ is the element $\cch(e)$ of the set $\Div_1(W)\times
\cP(I)^2$ defined through:
\be
\cch(e)=(s(e),I_x(e),I_y(e))~~.
\ee
This gives a map $\cch:\EF(R,W)\rightarrow \Div_1(W)\times \cP(I)^2$. 
\end{Definition} 

\

\begin{Lemma}
The even divisorial invariant $\cch:\Ob \EF(R,W)\rightarrow
\Div_1(W)\times \cP(I)^2$ is a {\em complete} even elementary
invariant.
\end{Lemma}
\begin{proof}
By Proposition \ref{prop:ssII}, two elementary factorizations of $W$
are isomorphic in $\hef(R,W)$ iff they have the same $(s,x,y)$, which in
turn is equivalent with coincidence of their even elementary invariants. 
\qed
\end{proof}

\

\noindent Using the essence $z(e)$ defined in \eqref{chs}, each
elementary factorization $e_v$ of $W$ determines an elementary
factorization $\essred(e)\eqdef e_{v/z(e)}$ of $\widetilde{W}\eqdef
W/z(e)^2$ (see Definition \ref{def:essred}).  For any normalized
critical divisor $z$ of $W$, let $\hef_z(R,W)$ denote the full
subcategory of $\hef(R,W)$ consisting of those elementary
factorizations whose essence equals $z$ and let $\ccHef_z(R,W)$ be its
set of isomorphism classes.

\

\begin{Lemma}
\label{lemma:chs2Isoc} For any critical divisor $z$ of $W$ such that
$(z,W/z^2)=(1)$, the map $\essred$ induces a well-defined bijection
$\check{e}ssred_z:\ccHef_z(R,W)\stackrel{\sim}{\rightarrow}
\ccHef_1(R,W/z^2)$.
\end{Lemma}
\begin{proof} The proof is almost identical to that of Lemma
\ref{lemma:chs2Iso}, but taking into account that in $\hef(R,W)$ we
deal only with the even morphisms of $\HEF(R,W)$. \qed
\end{proof}
Let:
\ben
\chS \eqdef \im\, h\subset \Div_1(W)\times \cP(I)^2~~.
\een
For a subset $J\subset I^\0$, let:
\be
\chS_J\eqdef \cch(\ccHef_{z_J}(R,W))\subset \chS~~,
\ee
where $z_J$ was defined in \eqref{f151}. Define $\chN_J(R,W)\eqdef
|\chS_{J}|$ and
$\chN_\emptyset(R,W)=|\cch(\ccHef_1(R,W))|=|\ccHef_1(R,W)|$, where the
last equality holds since $\cch$ is a complete even elementary
invariant. We can again compute $\chN(R,W)$ in terms of
$\chN_\emptyset(R,W)$:

\

\begin{Proposition}
\label{prop:N1toNc}
We have: 
\ben
\label{f183}
\chN(R,W) = \sum_{J\subset I^{\0}} \chN_\emptyset (R,W/z_J^2)~~.
\een
\end{Proposition}

\begin{proof}
Follows from Lemma \ref{lemma:chs2Isoc}. \qed
\end{proof}

\

\noindent Define:
\be
\chS_\emptyset=\big\{\cch(e) \,\big|\, e\in \Ob \HEF(R,W)~\mathrm{and}~  z(e)=1\big\}
\ee
and:
\be
\chS_{\emptyset,k}\eqdef \big\{\cch(e) \,\big|\, e\in \Ob \HEF(R,W)\,,~ z(e)=1~\mathrm{and}~ |I_s(e)|=k\big\}~~.
\ee

\begin{Lemma}
\label{lemma:N0c}
For $k\geq 1$, we have:
\ben\label{f120c}
|\chS_{\emptyset,k}|=2^{k}\cdot \sum_{\substack{K\subset I,\\ |K|=k}}\prod_{i\in K} \big\lfloor\frac{n_i-1}{2}\big\rfloor ~~.
\een
\end{Lemma}
\begin{proof} The proof is similar to that of Lemma \ref{lemma:N0}.
Consider a subset $K\subset I(s)$ of cardinality $k$. As in Lemma \ref{lemma:N0},
there are $\prod_{i\in K}\big\lfloor\frac{n_i-1}{2}\big\rfloor$ different
possibilities for $s$ such that $I(s)=K$. Fixing the set $I(s)$ with
$|I(s)|=k$, we have a set $\cP(I(s))$ of $2^k$ partitions
$I(s)=I_x\sqcup I_y$. Define:
\be
\chS_{\emptyset,k,s}=\big\{\cch(e) \,\big|\, e\in \Ob \,\hef(R,W)~ ,~~ z(e)=1~,~~ |I_s(e)|=k~~\mathrm{and}~~s(e)=s\, \big\}~~.
\ee
The map which sends a partition $\beta=(I_1,I_2)$ of $I(s)$ to the
element $\alpha=(s,I_1,I_2)$ is a bijection. We compute:
\be
|\chS_{\emptyset,k,s}|=|\cP(I(s))|=2^k~~.
\ee
This holds for any $s$ with $I(s)=K$, where $K\subset I$ has
cardinality $k$. Since $\chS_{\emptyset,k}=\sqcup_s \chS_{\emptyset
,k, s}$ and since the cardinality $|\chS_{\emptyset,k,s}|$ does not
depend on $s$, we find:
\be
|\chS_{\emptyset,k}|=\sum_s |\chS_{\emptyset ,k,s}|= 2^{k}\cdot \sum_{K\subset I: |K|=k}\prod_{i\in K} \big\lfloor\frac{n_i-1}{2}\big\rfloor ~~.
\ee 
\qed
\end{proof}

\

\noindent An immediate consequence is the following:

\

\begin{Proposition}
\label{prop:Ss_carc}
With the definitions above, we have: 
\ben \label{f126c}
\chN_\emptyset(R,W)=|\chS_\emptyset|=\sum_{k=0}^r 2^{k}\sum_{\substack{K\subset I\\|K|=k }}  \prod_{i\in K} \big\lfloor\frac{n_i-1}{2}\big\rfloor ~~.
\een
\end{Proposition}

\

\noindent We are now ready to compute $\chN(R,W)$.

\

\begin{Theorem}
\label{thm:Nc}
The number of isomorphism classes of the category $\hef(R,W)$ for a critically-finite $W$ as in \eqref{f129} is given by:
\ben
\label{Nfinalc}
\chN(R,W)=\sum_{k=0}^{r^\1}\sum_{\substack{K\subsetneq I, \\|K^\1|=k}} 2^{r^{\0}+k} \prod_{i\in K} \big\lfloor\frac{n_i-1}{2}\big\rfloor ~~,
\een
\end{Theorem}
\begin{proof}
Using Proposition \ref{prop:N1toNc} and Proposition \ref{prop:Ss_carc},
we write:
\ben
\label{f125c}
\chN(R,W) =   \sum_{J\subset I^{\0}} \big( \sum_{k=0}^{r-j} 2^{k}\sum_{\substack{K\subset I\backslash J\\|K|=k }}  \prod_{i\in K} \big\lfloor\frac{n_i-1}{2}\big\rfloor \big)~~,
\een
where $j\eqdef |J|$. 
Consider a subset
$K\subset I$ as in \eqref{f126c}. Such an index set $K=K^{\1}\sqcup
K^{\0}\subset I$ of cardinality $k=k^{\1} + k^{\0}$ appears in
$\chN_\emptyset(R,\widetilde{W})$ if $K^{\0}\subset I^{\0}\backslash J$. The
coefficient of $\prod_{i\in K} \big\lfloor\frac{n_i-1}{2}\big\rfloor$
in $\chN(R, \widetilde{W})$ is $2^{k}$ for any of $r^{\0}-k^{\0}
\choose j$ choices of $J$. It follows that this coefficient in
$\chN(R,W)$ is:
\ben\label{f128c}
\sum_{j=0}^{r^{\0}-k^{\0}} 2^{k}{r^{\0}-k^{\0} \choose j } = 2^{k^\0+k^\1} 2^{r^{\0}-k^{\0}}  \\ =
  2^{r^{\0}+k^{\1}}~~.
\een
Together with \eqref{f128c}, relation \eqref{f125c} yields \eqref{Nfinalc}. 
\qed
\end{proof}

\

\section{Some examples}

\noindent In this section, we discuss a few classes of
examples to which the results of the previous sections
apply. Subsection 5.1 considers the ring of complex-valued holomorphic
functions defined on a smooth, non-compact and connected Riemann
surface, which will be discussed in more detail in a separate
paper. Subsection 5.2 considers rings arising through the
Krull-Kaplansky-Jaffard-Ohm construction, which associates to any
lattice-ordered Abelian group a B\'ezout domain having that ordered
group as its group of divisibility. Subsection 5.3 discusses B\'ezout
domains with a specified spectral poset, examples of which can be
produced by a construction due to Lewis.

\subsection{Elementary holomorphic factorizations over a non-compact Riemann surface}

\noindent Let $\Sigma$ be any non-compact connected Riemann surface
(notice that such a surface need not be algebraic and that it may have
infinite genus and an infinite number of ends). It is known that the
cardinal Krull dimension of $\O(\Sigma)$ is independent of 
$\Sigma$ and is greater than or equal to $2^{\aleph_1}$
(see \cite{Henriksen2,Alling1}). The following classical result (see
\cite{Alling2,Remmert}) shows that the $\C$-{\em algebra} of
holomorphic functions entirely determines the complex geometry of
$\Sigma$.

\

\begin{Theorem}[Bers]
Let $\Sigma_1$ and $\Sigma_2$ be two connected non-compact Riemann
surfaces. Then $\Sigma_1$ and $\Sigma_2$ are biholomorphic iff
$\O(\Sigma_1)$ and $\O(\Sigma_2)$ are isomorphic {\em as
  $\C$-algebras}.
\end{Theorem}

\

\noindent A B\'ezout domain $R$ is called {\em adequate} if for any
$a\in R^\times$ and any $b\in R$, there exist $r,s\in R$ such that
$a=rs$, $(r,b)=(1)$ and any non-unit divisor $s'$ of $s$ satisfies
$(s',b)\neq (1)$. It is known that any adequate B\'ezout domain $R$ is
an elementary divisor domain, i.e.  any matrix with elements from $R$
admits a Hermite normal form.  The following result provides a large
class of examples of non-Noetherian adequate B\'ezout domains:

\

\begin{Theorem}
For any smooth and connected non-compact Riemann surface $\Sigma$, the
ring $\O(\Sigma)$ is an adequate B\'ezout domain and hence an
elementary divisor domain.
\end{Theorem}

\

\begin{proof} The case $\Sigma=\C$ was established in
\cite{Helmer2,Henriksen3}. This generalizes to any Riemann surface
using \cite{Alling1,Alling2}. Since $\O(\Sigma)$ is an adequate B\'ezout
domain, it is also a $PM^\ast$ ring\footnote{A {\em $PM^\ast$}-ring is
  a unital commutative ring $R$ which has the property that any
  non-zero prime ideal of $R$ is contained in a {\em unique} maximal
  ideal of $R$.} \cite{Helmer2} and hence \cite{Zabavsky} an
elementary divisor domain. The fact that $\O(\Sigma)$ is an elementary
divisor domain can also be seen as follows.  Guralnick \cite{G} proved
that $\O(\Sigma)$ is a B\'ezout domain of stable range one. By \cite{EO,R}, 
this implies that $\O(\Sigma)$ is an elementary divisor domain. \qed
\end{proof}

The prime elements of $\O(\Sigma)$ are those holomorphic functions
$f:\Sigma\rightarrow \C$ which have a single simple zero on
$\Sigma$. This follows, for example, from the Weierstrass
factorization theorem on non-compact Riemann surfaces (see
\cite[Theorem 26.7]{FO}).  A critically-finite element $W\in
\O(\Sigma)$ has the form $W=W_0W_c$, where $W_0:\Sigma\rightarrow \C$
is a holomorphic function with (possibly infinite) number of simple
zeroes and no multiple zeroes while $W_c$ is a holomorphic function
which has only a finite number of zeroes, all of which have
multiplicity at least two. All results of this paper apply to this
situation, allowing one to determine the homotopy category $\hef(R,W)$
of elementary D-branes (and to count the isomorphism classes of such)
in the corresponding {\em holomorphic} Landau-Ginzburg model
\cite{lg1,lg2} defined by $(\Sigma,W)$. 

\subsection{Constructions through the group of divisibility}
\label{subsec:div}
Recall that the {\em group of divisibility} $G(R)$ of an integral
domain $R$ is the quotient group $K^\times/U(R)$, where $K$ is the
quotient field of $R$ and $U(R)$ is the group of units of $R$. This is
a partially-ordered Abelian group when endowed with the order induced
by the $R$-divisibility relation, whose positive cone equals
$R^\times/U(R)$.  Equivalently, $G(R)$ is the group of principal
non-zero fractional ideals of $R$, ordered by reverse inclusion. Since
the positive cone generates $G(R)$, a theorem due to Clifford implies
that $G(R)$ is a directed group (see \cite[par. 4.3]{Mott}). It is an
open question to characterize those directed Abelian groups which
arise as groups of divisibility of integral domains. It is known that
$G(R)$ is totally-ordered iff $R$ is a valuation domain, in which case
$G(R)$ is order-isomorphic with the value group of $R$ and the natural
surjection of $K^\times$ to $G(R)$ gives the corresponding
valuation. Moreover, a theorem due to Krull \cite{Krull} states that
any totally-ordered Abelian group arises as the group of divisibility
of some valuation domain.  It is also known\footnote{Notice that a UFD
  is a B\'ezout domain iff it is Noetherian iff it is a PID (see
  Appendix \ref{app:Bezout}).} that $R$ is a UFD iff $G(R)$ is
order-isomorphic with a (generally infinite) direct sum of copies of
$\Z$ endowed with the product order (see \cite[Theorem 4.2.2]{Mott}).

An ordered group $(G,\leq)$ is called {\em lattice-ordered} if the
partially ordered set $(G,\leq)$ is a lattice, i.e. any two element
subset $\{x,y\}\subset G$ has an infimum $\inf(x,y)$ and a supremum
$\sup(x,y)$ (these two conditions are in fact equivalent for a group
order); in particular, any totally-ordered Abelian group is
lattice-ordered.  Any lattice-ordered Abelian group is torsion-free
(see \cite[p. 10]{Jaffard} or \cite[15.7]{Gilmer}). The divisibility
group $G(R)$ of an integral domain $R$ is lattice-ordered iff $R$ is a
GCD domain \cite{Mott}. In particular, the group of divisibility of a
B\'ezout domain is a lattice-ordered group.

When $R$ is a B\'ezout domain, the prime elements of $R$ are detected
by the lattice-order of $G(R)$ as follows. Given any Abelian
lattice-ordered group $(G,\leq)$ and any $x\in G$, let $\uparrow
x\eqdef \{y\in G|x\leq y\}$ and $\downarrow x\eqdef \{y\in G|y\leq
x\}$ denote the up and down sets determined by $x$. A {\em positive
  filter} of $(G,\leq)$ is a filter of the lattice $(G_+,\leq)$,
i.e. a proper subset $F\subset G_+$ having the following two
properties:
\begin{enumerate}[1.]
\itemsep 0.0em
\item $F$ is upward-closed, i.e. $x\in F$ implies $\uparrow x\subset F$
\item $F$ is closed under finite meets, i.e. $x,y\in F$ implies $\inf(x,y)\in F$. 
\end{enumerate}
Notice that $\uparrow x$ is a positive filter for any $x\in G_+$.  A
positive filter $F$ of $(G,\leq)$ is called:
\begin{enumerate}[(a)]
\itemsep 0.0em
\item {\em prime}, if $G_+\setminus F$ is a semigroup, i.e. if $x,y\in
  G_+\setminus F$ implies $x+y\in G_+\setminus F$.
\item {\em principal}, if there exists $x\in F$ such that $F=\uparrow
  x$.
\end{enumerate}
If $R$ is a B\'ezout domain with field of fractions $K$ and group of
divisibility $G=K^\times/U(R)$, then the natural projection
$\pi:K^\times\rightarrow G$ induces a bijection between the set of
proper ideals of $R$ and the set of positive filters of $G$, taking a
proper ideal $I$ to the positive filter $\pi(I\setminus \{0\})$ and
a positive filter $F$ to the proper ideal $\{0\}\cup \pi^{-1}(F)$
(see \cite{Sheldon,BrandalBezout}). This correspondence
maps prime ideals to prime positive filters and non-zero principal
ideals to principal positive filters. In particular, the prime
elements of $R$ correspond to the principal prime positive filters of
$G$.

The following result shows (see \cite[Theorem 5.3, p. 113]{FS}) that any
lattice-ordered Abelian group is the group of divisibility of some
B\'ezout domain, thus allowing one to construct a very large class of
examples of such domains using the theory of lattice-ordered groups:

\

\begin{Theorem}
\label{KKJO}(Krull-Kaplansky-Jaffard-Ohm)
If $(G,\leq)$ is a lattice-ordered Abelian group, then there exists a
B\'ezout domain $R$ whose group of divisibility is order-isomorphic to
$(G,\leq)$.
\end{Theorem}

\

\noindent For any totally ordered group $G_0$, the result of Krull
mentioned above gives a valuation ring whose divisibility group is
order-isomorphic to $G_0$. This valuation ring can be taken to be the
group ring $k[G_0]$, where $k$ is a field together with the following
valuation on the field of fractions $k(G_0)$:
\be
v\left(\sum_{i=1}^m a_i X_{g_i}/\sum_{j=1}^n b_j X_{h_j}\right)=\inf(g_1,\ldots, g_m) - \inf(h_1, \ldots, h_n)~~, 
\ee 
where it is assumed that all coefficients appearing in the expression
are non-zero. Lorenzen \cite{Lorenzen} proved that every
lattice-ordered group can be embedded into a direct product of totally
ordered groups with the product ordering. This embedding is used by
Kaplansky and Jaffard to construct the valuation domain $R$ of Theorem
\ref{KKJO}. By the result of Lorenzen, there exists a lattice
embedding $f: G \to H\eqdef\prod_{ \gamma \in \Gamma} G_{\gamma}$,
where $G_{\gamma}$ is a totally ordered group for all $\gamma \in
\Gamma$ and $H$ has the product ordering.  Let $Q=k( \{ X_g : g \in G
\})$ be the group field with coefficients in a field $k$ with the set
of formal variables $X_g$ indexed by elements of $G$. There is a
valuation $\varphi: Q^\times \to H$. The integral domain $R$ is the
domain defined by this valuation, i.e. $R\eqdef \{0\} \cup \{x \in
Q^\times : \varphi(x) \geq 0\}$. It is proved by Ohm that the
divisibility group of $R$ is order-isomorphic to $G$. Combining this
with the results of the previous sections, we have:

\

\begin{Proposition}
\label{prop:group}
Let $G$ be any lattice-ordered group and consider the ring $R$
constructed by Theorem \ref{KKJO}. Assume further that $W$ is a
critically-finite element of $R$. Then statements of Theorems
\ref{thm:Krull}, \ref{thm:decomp} and \ref{thm:N} hold.
\end{Proposition}

\

\noindent The simplest situation is when the lattice-ordered group $G$
is totally ordered, in which case $R$ is a valuation domain. Then a
proper subset $F\subset G_+$ is a positive filter iff it is
upward-closed, in which case the complement $G\setminus F$ is
non-empty and downward-closed. If $G\setminus F$ has a greatest
element $m$, then $G\setminus F=\downarrow m$ and $F=(\uparrow m)
\setminus \{m\}$. If $G\setminus m$ does not have a greatest element,
then $(G\setminus F, F)$ is a Dedekind cut of the totally-ordered set
$(G,\leq)$. For example we can take $G\in \{\Z,\QQ,\R\}$ with the
natural total order:
\begin{itemize}
\itemsep 0.0em
\item When $G=\Z$, the B\'ezout domain $R$ is a discrete valuation
  domain and thus a PID with a unique non-zero prime ideal and hence
  with a prime element $p\in R$ which is unique up to association in
  divisibility. In this case, any positive filter of $\Z$ is principal
  and there is only one prime filter, namely $\uparrow 1=\Z_+\setminus
  \{0\}$. A critically-finite potential has the form $W=W_0 p^k$,
  where $k\geq 2$ and $W_0$ is a unit of $R$.
\item When $G=\QQ$, there are two types of positive filters. The first
  have the form $F=(\uparrow q) \setminus \{q\}=(q,+\infty)\cap \QQ$
  with $q\in \QQ_{\geq 0}$, while the second correspond to Dedekind
  cuts and have the form $F=[a,+\infty)\cap \QQ$ with $a\in
    \R_{>0}$. In particular, a positive filter is principal iff it has
    the form $F=[q,+\infty)\cap \QQ$ with $q\in \QQ_{>0}$. A principal
      positive filter can never be prime, since $\QQ_+\setminus
      F=[0,q)\cap \QQ$ is not closed under addition for $q>0$. Hence
        the B\'ezout domain $R$ has no prime elements when
        $G=\QQ$.
\item When $G=\R$, any proper subset of $\R_+$ has an infimum and
  hence positive filters have the form $F=(a,+\infty)$ or
  $F=[a,+\infty)$ with $a\in \R_{>0}$, the latter being the principal
    positive filters. No principal positive filters can be prime, so
    the corresponding B\'ezout domain has no prime
    elements.
\end{itemize}

\noindent We can construct more interesting examples as follows. 
Let $(G_i, \leq_i)_{i\in I}$ be any family of lattice-ordered Abelian groups,
where the non-empty index set $I$ is arbitrary.  Then the direct
product group $G\eqdef \prod_{i\in I} G_i$ is a lattice-ordered Abelian group
when endowed with the product order $\leq$:
\be
(g_i)_{i\in I}\leq (g'_i)_{i\in I}~~\mathrm{iff}~~\forall i\in I: g_i\leq_i g'_i~~.
\ee
Let $\supp(g)\eqdef \{i\in I|g_i\neq 0\}$. The direct sum $G^0\eqdef
\oplus_{i\in I} G_i=\{g=(g_i)\in G||\supp(g)<\infty\}$ is a subgroup
of $G$ which becomes a lattice-ordered Abelian group when endowed with
the order induced by $\leq$.  For any $x=(x_i)_{i\in I}\in G$, we have
$\uparrow_G x=\prod_{i\in I} \uparrow x_i$ while for any $x^0\in G^0$,
we have $\uparrow_{G^0} x^0=\oplus_{i\in I}\uparrow x_i^0$, where
$\uparrow_G$ and $\uparrow_{G_0}$ denote respectively the upper sets
computed in $G$ and $G^0$.  Hence: 
\begin{enumerate}[I.]
\itemsep 0.0em
\item The principal positive filters of $G$ have the form
  $F=\prod_{i\in I}F_i$, where:
\begin{enumerate}[1.]
\itemsep 0.0em
\item each $F_i$ is a non-empty subset of $G_{i+}$ which either
  coincides with $\uparrow 0_i$ or is a principal positive filter of
  $(G_i,\leq_i)$
\item at least one of $F_i$ is a principal positive filter of $(G_i,\leq_i)$.
\end{enumerate}
Such a principal positive filter $F$ of $G$ is prime iff the set
$G_{i+}\setminus F_i$ is empty or a semigroup for all $i\in I$. In
particular, the principal prime ideals of the B\'ezout domain $R$
associated to $G$ by the construction of Theorem \ref{KKJO} are in
bijection with families (indexed by $I$) of principal prime ideals of
the B\'ezout domains $R_i$ associated to $G_i$ by the same
construction. The {\em non-zero} principal prime ideals of $R$ are in
bijection with families $(J_i)_{i\in S}$, where $S$ is a non-empty
subset of $I$ and $J_i$ is a non-zero principal prime ideal of $R_i$
for each $i\in S$.
\item The principal positive filters of $G^0$ have the form
  $F^0=\oplus_{i\in I}F_i$, where $(F_i)_{i\in I}$ is a family of
  subsets of $F_i\subseteq G_{i+}$ such that the set $\supp F\eqdef
  \{i\in I|F_i\neq \uparrow_{G_i} 0_i \}$ is finite and non-empty and
  such that $F_i$ is a principal positive filter of $(G_i,\leq_i)$ for
  any $i\in \supp F$.  Such a principal positive filter $F^0$ of $G^0$
  is prime iff $F_i$ is prime in $(G_i,\leq_i)$ for any $i\in \supp
  F$. In particular, the non-zero principal prime ideals of the B\'ezout
  domain $R^0$ associated to $G^0$ by the construction of Theorem
  \ref{KKJO} are in bijection with finite families of the form
  $J_{i_1},\ldots, J_{i_n}$ ($n\geq 1$), where $i_1,\ldots, i_n$ are
  distinct elements of $I$ and $J_{i_k}$ is a non-zero principal prime
  ideal of the B\'ezout domain $R_{i_k}$ associated to $G_{i_k}$ by the
  same construction.
\end{enumerate}

\noindent It is clear that this construction produces a very large
class of B\'ezout domains which have prime elements and hence to which
Proposition \ref{prop:group} applies. For example, consider the direct
power $G=\Z^I$ and the direct sum $G=\Z^{(I)}$, endowed with the
product order. Then the B\'ezout domain $R^0$ associated to $\Z^{(I)}$
is a UFD whose non-zero principal prime ideals are indexed by the {\em
  finite} non-empty subsets of $I$.  On the other hand, the non-zero
principal prime ideals of the B\'ezout domain $R$ associated to $\Z^I$
are indexed by {\em all} non-empty subsets of $I$. Notice that $R$ and
$R^0$ coincide when $I$ is a finite set.

\subsection{Constructions through the spectral poset}
Given a unital commutative ring $R$, its {\em spectral poset} is the
prime spectrum $\Spec(R)$ of $R$ viewed as a partially-ordered set
with respect to the order relation $\leq$ given by inclusion between
prime ideals. Given a poset $(X,\leq)$ and two elements $x,y\in X$, we
write $x\ll y$ if $x<y$ and $x$ is an immediate neighbor of $y$, i.e.
if there does not exist any element $z\in X$ such that $x<z<y$. It was
shown in \cite{KaplanskyBook} that the spectral poset of any unital
commutative ring satisfies the following two conditions, known as {\em
  Kaplansky's conditions}:
\begin{enumerate}[I.]
\itemsep 0.0em
\item Every non-empty totally-ordered subset of $(\Spec(R),\leq)$ has
  a supremum and an infimum (in particular, $\leq$ is a lattice
  order).
\item Given any elements $x,y\in \Spec(R)$ such that $x<y$, there
  exist distinct elements $x_1,y_1$ of $\Spec(R)$ such that $x\leq
  x_1<y_1\leq y$ and such that $x_1\ll y_1$.
\end{enumerate}
It is known \cite{Hochster,Lewis} that these conditions are not sufficient
to characterizes spectral posets.  It was shown in \cite{Speed} that a
poset $(X,\leq)$ is order-isomorphic with the spectral poset of a
unital commutative ring iff $(X,\leq)$ is profinite, i.e. iff
$(X,\leq)$ is an inverse limit of finite posets; in particular, any finite 
poset is order-isomorphic with a spectral poset \cite{Lewis}. 

A partially ordered set $(X,\leq)$ is called a {\em tree} if for
every $x\in X$, the lower set $\downarrow x=\{y\in X|y \leq x\}$ is
totally ordered. The following result was proved by Lewis:

\

\begin{Theorem}{\rm \cite{Lewis}}
\label{thm:Lewis}
Let $(X,\leq)$ be a partially-ordered set. Then the following
statements are equivalent:
\begin{enumerate}[(a)]
\itemsep 0.0em
\item $(X,\leq)$ is a tree which has a unique minimal element $\theta\in X$ and
  satisfies Kaplansky's conditions I. and II.
\item $(X,\leq)$ is isomorphic with the spectral poset of a B\'ezout
  domain.
\end{enumerate}
Moreover, $R$ is a valuation domain iff $(X,\leq)$ is a totally-ordered set. 
\end{Theorem}

\

\noindent An explicit B\'ezout domain $R$ whose spectral poset is
order-isomorphic with a tree $(X,\leq)$ satisfying condition (a) of
Theorem \ref{thm:Lewis} is found by first constructing a
lattice-ordered Abelian group $G$ associated to $(X,\leq)$ and then
constructing $R$ from $G$ is in Theorem \ref{KKJO}. 
The lattice-ordered group $G$ is given by \cite{Lewis}:
\be
G=\{f:X^\ast\rightarrow \Z\, |\, |\supp(f)|<\infty\}~~,
\ee
where $X^\ast\eqdef \{x\in X\, | \, \exists y\in X: y\ll x\}$ and
$\supp f\eqdef \{x\in X^\ast \, | \, f(x)\neq 0\}$ is a tree when
endowed with the order induced from $X$. The group operation is given
by pointwise addition. The lattice order on $G$ is defined by the
positive cone:
\ben
\label{poscone}
G_+\eqdef \{f\in G\, | \, f(x)>0 ~\forall x\in \minsupp(f)\}=\{f\in G \, | \, f(x)\geq 1 ~\forall x\in \minsupp(f)\}~~,
\een
where the order on $\Z$ is the natural order and the {\em minimal
  support} of $f\in G$ is defined through:
\ben
\label{minsupp}
\minsupp(f)\eqdef \{x\in \supp(f) \,|\, \forall y\in X^\ast ~\mathrm{such~that}~y<x: f(y)=0\}~~.
\een
Notice that $f\in G_+$ if $\minsupp(f)=\emptyset$ (in particular, we
have $0\in G_+$).  The lattice-ordered Abelian group $G$ has the
property that the set of its prime positive filters\footnote{Called
  ``prime V-segments'' in \cite{Lewis}.} (ordered by inclusion) is
order-isomorphic with the tree obtained from $(X,\leq)$ by removing
the minimal element $\theta$ (which corresponds to the zero ideal of
$R$). Explicitly, the positive prime filter $F_x$ associated to an
element $x\in X\setminus \{\theta\}$ is defined through
\cite[p. 432]{Lewis}:
\ben
\label{Fx}
F_x\eqdef \{f\in G_+ \, | \, \exists y\in \minsupp(f): y\leq
x\}=\{f\in G_+ \, | \, \minsupp(f)\cap (\downarrow x)\neq
\emptyset\}~~.
\een 
By Lemma \ref{lemma:prime}, a principal prime ideal of a B\'ezout domain is
necessarily maximal. This implies that the prime elements of $R$
(considered up to association in divisibility) correspond to certain
maximal elements of the tree $(X,\leq)$. Notice, however, that a
B\'ezout domain can have maximal ideals which are not principal (for
example, the so-called ``free maximal ideals'' of the ring of
complex-valued holomorphic functions defined on a non-compact Riemann
surface $\Sigma$ \cite{Alling1}). For any maximal element $x$ of $X$
which belongs to $X^\ast$, let $1_x\in G$ be the element defined by
the characteristic function of the set $\{x\}$ in $X^\ast$:
\be 
1_x(y)\eqdef \twopartdef{1}{y=x}{0}{y\in X^\ast \setminus \{x\} }~~.  
\ee 
Then $\supp(1_x)=\minsupp(1_x)=\{x\}$ and $1_x\in G_+\setminus
\{0\}$. Notice that $1_x\in F_x$. 

\

\begin{Proposition}
\label{prop:spec}
Let $(X,\leq)$ be a tree which has a unique minimal element and
satisfies Kaplansky's conditions I. and II. and let $R$ be the B\'ezout
domain determined by $(X,\leq)$ as explained above.
\begin{enumerate}[(a)] \itemsep 0.0em
\item For each maximal element $x$ of $X$ which belongs to $X^\ast$, the
principal positive filter $\uparrow 1_x$ is prime and hence
corresponds to a principal prime ideal of $R$. Moreover, we have:
\ben
\label{up1x}
\uparrow 1_x=\{f\in G_+ \, | \, \supp(f)\cap \downarrow x\neq \emptyset\}
\een
and:
\ben
\label{Fx1x}
F_x=\{f\in \, \uparrow 1_x \, | \, \inf S_f(x) \in S_f(x)\}=\{f\in \, \uparrow 1_x \, | \, \exists \min S_f(x)\}~~,
\een
where: 
\be
S_f(x)\eqdef \supp(f)\cap \downarrow x~~.
\ee
\item Let $W$ be a critically-finite element of $R$.  Then the
  statements of Theorem \ref{thm:Krull}, \ref{thm:decomp} and \ref{thm:N}
  hold.
\end{enumerate}
\end{Proposition}

\begin{proof}
The second statement follows from the results of the previous
sections. To prove the first statement, let $x$ be a maximal element
of $X$ which belongs to $X^\ast$. We have:
\ben
\label{up1}
\uparrow 1_x=\{f\in G_+ \, | \, f-1_x\in G_+\}=\{f\in G_+ \, | \, f(y)> 1_x(y)~\forall y\in \minsupp(f-1_x) \}~~.
\een
On the other hand, we have $\minsupp(f-1_x)=\{y\in X^\ast \, | \,
f(y)\neq 1_x(y) \, \& \, \forall z\in X^\ast ~\mathrm{such~that}~ z<
y: f(z)=1_x(z)\}$. Since $x$ is maximal, any element $z\in X^\ast$ for
which there exists $y\in X^\ast$ such that $z<y$ satisfies $z\neq
x$ and hence $1_x(z)=0$. This gives:
\beqa
\minsupp(f-1_x)&=&\{y\in X^\ast \, | \, f(y)\neq 1_x(y) \, \& \, \forall z\in X^\ast ~\mathrm{such~that}~ z< y: f(z)=0\}=\\
&=&\twopartdef{\minsupp(f)\cup \{x\}}{f\in Q_x}{\minsupp(f)\setminus \{x\}}{f\in G_+\setminus Q_x}~~,
\eeqa
where: 
\be
Q_x \eqdef \{f\in G_+ \, | \, f(x)\neq 1 \, \& \, \forall z\in X^\ast  ~\mathrm{such~that}~ z< x: f(z)=0\}=A_x \sqcup B_x~~,
\ee
with:
\beqa
A_x &\eqdef& \{f\in G_+ \, | \,  \forall z\in X^\ast  ~\mathrm{such~that}~ z\leq x: f(z)=0\}=\{f\in G_+ \, | \, \supp(f)\cap \downarrow x= \emptyset\}~~\nn\\
B_x &\eqdef& \{f\in G_+ \, | \, x\in \minsupp(f) \, \& \, f(x)>1\}\subset F_x\subset G_+\setminus A_x~~.
\eeqa
This gives: 
\ben
\uparrow 1_x=(G_+\setminus Q_x)\cup B_x =(G_+\setminus A_x)\cup B_x=G_+\setminus A_x=\{f\in G_+ \, | \, \supp(f)\cap \downarrow x\neq \emptyset\}~~,
\een
which establishes \eqref{up1x}. 
Notice that $G_+\setminus (\uparrow 1_x)=A_x$ is a semigroup, so
$\uparrow 1_x$ is a prime principal positive filter and hence it corresponds to
a principal prime ideal of $R$. Also notice that $F_x\subset \, \uparrow 1_x$. 

Consider an element $f\in \, \uparrow 1_x$. Then the non-empty set
$S_f(x)\eqdef \supp(f)\cap \downarrow x$ is totally ordered (since $X$
is a tree and hence $\downarrow x$ is totally ordered).  By
Kaplansky's condition I., this set has an infimum which we denote by
$x_f=\inf S_f(x)$; notice that $x_f\in \, \downarrow x$. For any $y\in
X^\ast$ with $y<x_f$, we have $y\not\in S_f(x)$ and hence $f(y)=0$.
Hence if $x_f$ belongs to $S_f(x)$ (i.e. if $S_f(x)$ has a minimum),
then $x_f=\min S_f(x)$ is an element of $\minsupp(f)\cap \downarrow x$
and in this case we have $f\in F_x$.  Conversely, given any element
$f\in F_x$, it is easy to see that the totally-ordered set
$\minsupp(f)\cap \downarrow x$ must be a singleton, hence
$\minsupp(f)\cap \downarrow x=\{x_f\}$ for a unique element $x_f\in
S_f(x)$.  This element must be a minimum (and hence an infimum) of the
totally-ordered set $S_f(x)$, since $x_f$ belongs to $\minsupp (f)$. We
conclude that \eqref{Fx1x} holds.
\qed
\end{proof}

\begin{remark}  
Statement (a) of Proposition \ref{prop:spec}
allows us to construct particular critically-finite elements of $R$ as
follows. For each maximal element of $X$ which belongs to $X^\ast$,
let $p_x$ be prime element of $R$ which generates the principal prime
ideal corresponding to the principal prime positive filter $\uparrow
1_x$ (notice that $p_x$ is determined up to association in
divisibility). For any finite collection $x_1,\ldots, x_N$ ($N\geq 1$)
of maximal elements of $X$ which belong to $X^\ast$ and any integers
$n_1,\ldots, n_N$ such that $n_j\geq 2$ for each $j\in \{1,\ldots,
N\}$, the element $W=\prod_{j=1}^N p_{x_j}^{n_j}\in R$ is
critically-finite.
\end{remark}

\

\noindent The following statement will be used in the construction of some examples below: 

\

\begin{Proposition} 
\label{prop:wo} 
Let $(S,\leq)$ be a well-ordered set. Then $(S,\leq)$ is a tree with a
unique minimal element.  Moreover, $(S,\leq)$ satisfies Kaplansky's
conditions I. and II. iff $S$ has a maximum. 
\end{Proposition}

\begin{proof} Since $S$ is well-ordered, it is totally ordered and has
a minimum, therefore it is a tree with a unique minimal element.
Given $x,y\in S$ such that $x<y$, we have $x\leq x_1\ll y_1\leq y$,
where $x_1\eqdef \min\{x<s\leq y\}$ and $y_1\eqdef \min\{x_1<s\leq
y\}$. Thus $S$ satisfies Kaplansky's condition $I.$ Any non-empty
totally-ordered subset $A\subset S$ has a minimum since $S$ is
well-ordered. Moreover, $A$ has a supremum (namely $\min\{s\in S \, |
\, \forall x\in A: x\leq s\}$) iff it has an upper bound. Hence $S$
satisfies Kaplansky's condition II. iff it every non-empty subset of
$S$ has an upper bound, which amounts to the condition that $S$ has a
greatest element. \qed
\end{proof}

\

\begin{remark} Every element of $S$ (except a possible greatest
element) has an immediate successor (upper neighbor). In particular,
$S$ has a maximal element iff it has a maximum $M$, which in turn
happens iff the order type $\alpha$ of $S$ is a successor ordinal. In
this case, $M$ has a predecessor iff $\alpha$ is a double successor
ordinal, i.e. iff there exists an ordinal $\beta$ such that
$\alpha=\beta+2$.
\end{remark}

\

\begin{Example}
\label{ex:tree1}
Consider the tree $T$ whose underlying set is the set $\N=\Z_{\geq 0}$
of non-negative integers together with the following partial order: $0
< n$ for every $n \in \N$ and there is no further strict inequality;
notice that any maximal vertex $n \in \N^\ast=\Z_{>0}$ has an
immediate lower neighbor, namely $0$. This corresponds to a countable
corolla, i.e. a tree rooted at $0$ and with an edge connecting the
root to $n$ for every $n \in \N^\ast$ (and no other edges).  By
Proposition \ref{prop:spec}, each maximal vertex $n \in \N^\ast$
corresponds to a principal prime ideal of the associated B\'ezout
domain.
\end{Example}

\

\begin{Example}
\label{ex:tree2} We can make the previous example more interesting by
replacing the edges of $T$ with a tree. For each $x \in \N^\ast$,
consider a tree $T_x$ with a unique root (minimal element) $r_x\in
T_x$ and which satisfies Kaplansky's conditions I. and II.  Consider
the tree $\cT$ obtained by connecting $0$ to $r_x$ for $x \in
\N^\ast$.  Then $T$ has a unique minimal element (namely $0$) and
satisfies Kaplansky's conditions I. and II. By Proposition
\ref{prop:spec}, those maximal elements of each of the trees $T_x$
which have an immediate lower neighbor correspond to prime elements of
the associated B\'ezout domain $R$. We obtain many examples of B\'ezout
domains by varying the trees $T_x$:

\begin{enumerate}[1.]  \itemsep 0.0em
\item Assume that for every $x \in \N^\ast$, the tree $T_x$ is reduced
to the single point $r_x=x$. Then we recover Example \ref{ex:tree1}.
\item For any element $x \in \N^\ast$, consider a finite tree $T_x$
and let $\Sigma_x$ be the set of maximal elements of $T_x$.  Then
$\cT^\ast= \cT\setminus \{0\}$ and any maximal element of $\cT$
different from $0$ has an immediate lower neighbor.  The corresponding
B\'ezout domain $R$ has a principal prime ideal for every element of
the set $\cup_{x\in \N^\ast} \Sigma_x$.
\item For each $x \in \N^\ast$, consider a well-ordered set $S_x$
which has a maximum $m_x$ and denote the minimum element of $S_x$ by
$r_x$. By Proposition \ref{prop:wo}, we can take $T_x=S_x$ in the
general construction above, thus obtaining a tree $\cT$ and a
corresponding B\'ezout domain $R$. Let $U\subset \N^\ast$ be the set
of those $x\in \N^\ast$ for which $S_x$ is a double successor
ordinal. Then each element of $U$ corresponds to a principal prime
ideal of $R$.
\end{enumerate}
\end{Example}

\appendix

\section{GCD domains}
\label{app:GCD} 

\noindent Let $R$ be an integral domain and $U(R)$ its multiplicative
group of units. For any finite sequence of elements $f_1,\ldots,
f_n\in R$, let $\langle f_1,\ldots, f_n\rangle$ denote the ideal
generated by the set $\{f_1,\ldots, f_n\}$. An element $u\in R$ is a
unit iff $\langle u\rangle=R$. Two elements $f,g\in R$ are called {\em
  associated in divisibility} (we write $f\sim g$) if there exists
$u\in U(R)$ such that $g=uf$.  This is equivalent with the condition
$\langle f\rangle=\langle g\rangle$. The association relation is an
equivalence relation on $R$.

\

\begin{Definition}
An integral domain $R$ is called a {\em GCD domain} if any two
elements $f,g$ admit a greatest common divisor (gcd). 
\end{Definition}

\

\noindent Let $R$ be a GCD domain. In this case, the gcd of two elements $f,g$
is determined up to association and the corresponding equivalence
class is denoted by $(f,g)$. Any two elements $f,g$ of $R$ 
also admit a least common multiple (l.c.m.), which is determined up
to association and whose equivalence class is denoted by
$[f,g]$. By induction, any finite collection of elements
$f_1,\ldots, f_n$ admits a gcd and and lcm, both of which are
determined up to association and whose equivalence classes are denoted by: 
\be
(f_1,\ldots, f_n)~~\mathrm{and}~~[f_1,\ldots, f_n]~~.
\ee

\begin{remark}
Any irreducible element of a GCD domain is prime, hence primes and
irreducibles coincide in a GCD domain. In particular, any element of a
GCD domain which can be factored into primes has unique prime factorization,
up to permutation and association of the prime factors.
\end{remark}

\section{B\'ezout domains}
\label{app:Bezout}

\noindent Let $R$ be a GCD domain. We say that the {\em B\'ezout
  identity} holds for two elements $f$ and $g$ of $R$ if for one
(equivalently, for any) gcd $d$ of $f$ and $g$, there exist $a,b\in R$
such that $d=af+bg$. This amounts to the condition that the ideal
$\langle f,g\rangle$ is principal, namely we have $\langle f,g\rangle
=\langle d\rangle$.

\subsection{Definition and basic properties}

\

\

\begin{Definition}
An integral domain $R$ is called a \emph{B\'ezout domain} if any (and
hence all) of the following equivalent conditions hold:
\begin{itemize}
\itemsep 0.0em
\item $R$ is a GCD domain and the B\'ezout identity holds for any two
  non-zero elements $f,g\in R$.
\item The ideal generated by any two elements of $R$ is principal.
\item Any finitely-generated ideal of $R$ is principal.
\end{itemize}
\end{Definition}

\noindent More generally, a {\em B\'ezout ring} is a unital commutative
ring $R$ which has the property that its finitely-generated ideals are
principal. Hence a B\'ezout domain is a B\'ezout ring which is an integral
domain.  The following well-known statement shows that the B\'ezout
property is preserved under quotienting by principal ideals:

\

\begin{Proposition}
Let $R$ be a B\'ezout ring and $I$ be a finitely-generated (hence
principal) ideal of $R$. Then $R/I$ is a B\'ezout ring.
\end{Proposition}

\

\noindent If $R$ is a B\'ezout domain and $f_1,\ldots, f_n\in R$, then
we have $\langle f_1,\ldots, f_n\rangle=\langle d\rangle$ for any
$d\in (f_1,\ldots, f_n)$ and there exist $a_1,\ldots, a_n\in R$ such
that $d=a_1f_1+\ldots+a_nf_n$. The elements $f_1,\ldots, f_n$ are
called {\em coprime} if $(f_1,\ldots, f_n)=(1)$, which amounts to the
condition $\langle f_1,\ldots, f_n\rangle=R$. This happens iff there
exist elements $a_1,\ldots, a_n\in R$ such that
$a_1f_1+\ldots+a_nf_n=1$. Notice that every B{\'e}zout domain is
integrally closed \cite{Bourbaki}.

\begin{remark}
B\'ezout domains coincide with those Pr\"ufer domains which are GCD
domains. Since any Pr\"ufer domain is coherent, it follows that any
B\'ezout domain is a coherent ring.
\end{remark}

\

\noindent The following result characterizes finitely-generated
projective modules over B\'ezout domains:

\

\begin{Proposition}
{\rm \cite{FS}}
Every finitely-generated projective module over a B\'ezout domain is free.
\end{Proposition}

\

\noindent In particular, finitely-generated projective factorizations
over a B\'ezout domain coincide with finite-rank matrix factorizations.

\subsection{Examples of B\'ezout domains}

\noindent The following rings are B\'ezout domains:

\begin{itemize}
\item Principal ideal domains (PIDs) coincide with the Noetherian
  B\'ezout domains. Other characterizations of PIDs among B\'ezout
  domains are given below.
\item Any generalized valuation domain is a B\'ezout domain.
\item The ring $\O(\Sigma)$ of holomorphic complex-valued functions
  defined on any\footnote{Notice that $\Sigma$ need not be
    algebraic. In particular, $\Sigma$ can have infinite genus and an
    infinite number of ends.} smooth connected non-compact Riemann
  surface $\Sigma$ is a non-Noetherian B\'ezout domain. In particular,
  the ring $\O(\C)$ of entire functions is a non-Noetherian B\'ezout
  domain.
\item The ring $\A$ of all algebraic integers (the integral closure of
  $\Z$ inside $\C$) is a non-Noetherian B\'ezout domain which has no
  prime elements.
\end{itemize}

\subsection{The Noetherian case}

\noindent The following is well-known:

\

\begin{Proposition}
Let $R$ be a B\'ezout domain. Then the following statements are
equivalent:
\begin{itemize}
\itemsep 0.0em
\item $R$ is Noetherian
\item $R$ is a principal ideal domain (PID)
\item $R$ is a unique factorization domain (UFD)
\item $R$ satisfies the ascending chain condition on principal ideals
  (ACCP)
\item $R$ is an atomic domain.
\end{itemize}
\end{Proposition}

\subsection{Characterizations of B\'ezout domains}

\

\

\begin{Definition}
Let $R$ be a commutative ring. The \emph{Bass stable rank} $\bsr(R)$
of $R$ is the smallest integer $n$, such that for any collection
$\{a_0, a_1, \dots, a_n\}$ of generators of the unit ideal, there
exists a collection $\{ \lambda_1, \dots , \lambda_n\}$ in $R$ such
that the collection $\{a_i- \lambda_i a_0: \ 1 \leq i \leq n\}$ also
generate the unit ideal. If no such $n$ exists, then $\bsr(R)\eqdef \infty$. 
\end{Definition}

\

\begin{Definition}
A unital commutative ring $R$ is called a {\em Hermite ring} (in the
sense of Kaplansky) if every matrix $A$ over $R$ is equivalent with an
upper or a lower triangular matrix.
\end{Definition}

\

\noindent The following result is proved in \cite[Theorem 8.1]{Mortini}

\

\begin{Theorem}
{\rm \cite{Mortini}}
Let R be a B\'ezout domain. Then $\bsr(R)\leq 2$. Moreover, $R$ is a
Hermite ring.
\end{Theorem}

\

\begin{acknowledgements}
This work was supported by the research grant IBS-R003-S1. 
\end{acknowledgements}

\

\end{document}